\documentclass[reqno]{amsart}

\usepackage{accents}
\usepackage{epsfig}
\usepackage{color}
\usepackage{amsthm}
\usepackage{amssymb}
\usepackage{amsmath,amssymb}
\usepackage{amscd}
\usepackage{amsopn}
\usepackage{changepage}
\usepackage{tikz-cd}
\usepackage{xcolor}
\usepackage{hyperref}
\usepackage{amsrefs}
\hypersetup{
    colorlinks=true,
    linkcolor=blue,
    citecolor=violet,
    filecolor=magenta,      
    urlcolor=cyan,
}

\numberwithin{equation}{section}
\newtheorem {Theorem}{Theorem}
\numberwithin{Theorem}{section}

\newtheorem {Lemma}[Theorem] {Lemma}
\newtheorem {Claim}[Theorem]{Claim}

\newtheorem {Proposition}[Theorem]{Proposition}
\newtheorem {Corollary}[Theorem]{Corollary}

\newtheorem{Remark}[Theorem]{Remark}

\tikzset{
    labl/.style={anchor=south, rotate=30, inner sep=1mm}
}
\tikzset{
    labbl/.style={anchor=south, rotate=330, inner sep=1mm}
}

\title[Local Symplectic Homology of Reeb Orbits]{Local Symplectic Homology of Reeb Orbits}
\author[Elijah Fender]{Elijah Fender}

\address{EF: Department of Mathematics, UC Santa Cruz, Santa Cruz, CA
  95064, USA} \email{fender@ucsc.edu}


\begin{document}

\begin{abstract}

In this paper we prove two isomorphisms in the local symplectic homology of a non iterated, isolated Reeb orbit. The isomorphisms are in $S^1$-equivariant and nonequivariant symplectic homology relating the local Floer homology group of the orbit to that of the return map. The $S^1$-equivariant symplectic homology isomorphism  can be stated succinctly as the local $S^1$-equivariant symplectic homology of an uniterated isolated Reeb orbit is isomorphic to the local Hamiltonian Floer homology of the return map. An  application of this result gives the equivalence of two different definitions of a Reeb orbit being a symplectically degenerate maximum.\\*

\end{abstract}

\maketitle


\tableofcontents


\section{Introduction and main results}
\label{sec:intro and results}

\subsection{Introduction}
\label{sec:intro}

The goal of this paper is to prove two isomorphisms in the local symplectic homology of a simple, which is to say non iterated, isolated Reeb orbit. The isomorphisms are in $S^1$-equivariant and nonequivariant symplectic homology relating the local Floer homology group of the orbit to that of the return map. The isomorphism we prove in $S^1$-equivariant symplectic homology can be stated succinctly as the local $S^1$-equivariant symplectic homology of a simple isolated Reeb orbit is isomorphic to the local Hamiltonian Floer homology of the return map. An immediate application of this result is the equivalence of two different definitions of a Reeb orbit being a symplectically degenerate maximum which we expand upon in Section \ref{sec:SDM}.\\

Results of this type are of interest for their role in proving the existence and multiplicity of Reeb orbits. The $S^1$-equivariant side of this paper can be viewed as a step in proving the following statement. Let $x=y^k$ be the k-th iteration of a simple Reeb orbit $y$ and let $\phi$ be the return map associated to the simple orbit $y$; then conjecturally \\*

\[
\begin{tikzcd}
SH^{S^1}(x) \arrow[dr,"\cong" labbl] \arrow[rr,"\cong"] && HF(\phi)^{\mathbb Z_k}\\
 & HC(x) \arrow[ur,"\cong" labl] 
\end{tikzcd}
\]

 \noindent where every arrow in this diagram is an isomorphism. Here $SH^{S^1}(x)$ denotes the $S^1$-equivariant local symplectic homology of $x$, $HC(x)$ denotes the local contact homology of $x$, and $HF(\phi)^{\mathbb Z_k}$ is the $\mathbb Z_k$-invariant part of the local Floer homology of $\phi$ with respect to the natural $\mathbb Z_k$ action. Note that there are serious conceptual issues with defining $HC(x)$ when $k>1$.\\
 
 In this paper Theorem \ref{thm:iso2} proves the simple case, that is when $k=1$, of the isomorphism $SH^{S^1}(x) \cong HF(\phi)$ without passing through contact homology. The corresponding simple case for contact homology, that is $HC(x)\cong HF(\phi)$, was first proved by Hryniewicz and Macarini in \cite{HM:contacthom}. Local contact homology was, in fact, defined in \cite{HM:contacthom}. In their paper, Hryniewicz and Macarini lay some of the ground work we build upon here as they explicitly give a local model for a neighborhood of a simple Reeb orbit which we use as the starting point of our computations. They also show that the dimension of Hamiltonian Floer homology of the return map gives an upper bound for the rank of the local contact homology of an iterated Reeb orbit, when it is defined. This work is continued in \cite{GHHM}, where Ginzburg, Hein, Hryniewicz, and Macarini conjectured that in fact for an iterated orbit we have $HC(x) \cong HF(\phi)^{\mathbb Z_k}$. They apply this to show that the existence of a Reeb orbit whose return map is a symplectically degenerate maximum guarantees the existence of infinitely many Reeb orbits. The results proved in \cite{GHHM} are further expanded upon In \cite{GGM}, where Ginzburg, G{\"u}rel and Macarini prove a Conley conjecture result for certain classes of contact manifolds using the conjectured isomorphism.\\

Several similar results in local homology have been proven to date.  Local symplectic homology was defined in \cite{GG:convex}, where Ginzburg and G{\"u}rel develop a Lusternik–Schnirelmann theory for equivariant symplectic homology. Finally in \cite{HHM:local Morse}, Hein, Hryniewicz, and Macarini show the transversality results which are necessary to define a local Morse complex with finite cyclic group symmetries. They apply this to the case of contact geometry in order to define a local contact homology with discrete action functionals. \\

On the nonequivariant side Theorem \ref{thm:iso1} can be seen as refining the simple case of a result proved by McLean in Lemma 3.4 from his paper \cite{Mc}. Here McLean proved an inequality involving rank of the symplectic homology of a Reeb orbit and the local Hamiltonian Floer homology of the return map direct summed with itself with a shift in degree. We show that the homology groups in question are in fact isomorphic when the orbit is simple.\\

\subsection{Results}
\label{sec:results}

In this section we introduce the setting of the problem and then present the main results proved in this paper. Let $(Y,\alpha)$ be a contact manifold. This is a manifold of dimension $2n+1$ with a one-form $\alpha$ such that $\alpha \wedge (d\alpha)^n\neq 0$. The Reeb vector field of $\alpha$ is a vector field $R$ such that $\iota_R\alpha =1$ and $\iota_Rd\alpha =0$. We say that $\psi^t$ is the Reeb flow of $\alpha$ on $Y$ if $\frac{\partial}{\partial t} \psi^t(p) = R_{\psi^t(p)}$ for $p \in Y$. \\

A Reeb trajectory $\gamma: \mathbb R \to Y$,$\gamma(t) =\psi^t(p)$ for $ p\in Y$ is a $T$-periodic closed Reeb orbit if $\gamma (t+T) = \gamma(t)$. We may refer to a closed Reeb orbit as simply a Reeb orbit. We say $\gamma$ is a {\it simple} $T$-periodic closed Reeb orbit if $\gamma$ is injective on the interval $[0,T)$. In particular, when we say a simple Reeb orbit, we mean an uniterated Reeb orbit. We say a $T$-periodic closed Reeb orbit is {\it isolated} if $\gamma$ is a period-$T$ closed Reeb orbit and there is some open neighborhood $\Sigma = S^1\times B^{2n}$ of $\gamma$ in $Y$, where $\gamma=S^1\times 0$, and $\Sigma$ contains no closed Reeb orbits of period-$t$ where $t\in (T-\eta, T+\eta)$, for some $\eta>0$. Note that we cannot exclude all possible periods here because traversing $\gamma$ twice, that is the orbit $\gamma^2$,  is another closed Reeb orbit of period $2T$ that lies in $\Sigma$. Furthermore, an arbitrary Reeb orbit or its iterates might not be isolated from other closed orbits broadly speaking. From this point forward we will consider $\gamma$ to be an isolated simple closed Reeb orbit unless otherwise specified.\\

The simple Reeb orbit $\gamma$ has several homology groups which one can associate with it. In this paper we explore the relationship between two specific types of associated homology groups, and what implication this relationship has for the dynamics of $\alpha$. The two types of homology groups associated to $\gamma $ in which we are interested are the symplectic homology of $\gamma$ and the local Hamiltonian Floer homology of the return map of the Reeb flow restricted to a slice in the trivialization $\Sigma =S^1 \times B^{2n}$.\\

Included in the choice for $\Sigma$ we also fixed a trivialization $S^1\times B^{2n}$. On this trivialization, the two-form $d\alpha$ restricts to a symplectic form on any slice $\theta \times B^{2n}$. Let $\phi$ be the return map obtained by following the Reeb flow from a small ball in the slice $0\times B^{2n}$ to a slightly larger ball in that same slice. Then $\phi$ is necessarily the germ of a symplectomorphism of $(B^{2n},d\alpha|_{0 \times B^{2n}})$ fixing the origin, which is discussed in \cite{HM:contacthom} as well as \cite{GG:gap} in Remark 4.1.\\

The fact that $\gamma$ is an isolated Reeb orbit is equivalent to $\phi$ having an isolated fixed point at the origin, this fact will be elaborated upon in Remark \ref{H fixes origin}. As a result, we know that $\phi$ has a  local Hamiltonian Floer homology defined in a neighborhood of the origin. Furthermore, if we were to look at the return map for any other slice $\theta_0 \times B^{2n}$, it would be a Hamiltonian diffeomorphism which would be conjugate to $\phi$, and therefore it would compute the same local Hamiltonian Floer homology on $B^{2n}$. So we can unambiguously (up to a choice of isotopy) associate the Hamiltonian Floer homology of $\phi$ with the orbit $\gamma$.\\

On the other hand, we may use a different Hamiltonian Floer Homology to study the dynamics near $\gamma$. We can produce a symplectization $W=(1-\delta,1+\delta) \times \Sigma$ of $\Sigma$ with coordinates $(r,\theta, z)$ and the symplectic form $\omega=d(r\alpha)$. Next, take a Hamiltonian $f(r)$ on $W$ which only depends on $r$ and has the following properties: $f(1)=f'(1)=1, f''>0$. With these conditions assumed for $f$ one can show, as we do in Lemma \ref{thm:one_periodic_are_reeb}, that the fixed points of the time-one flow of $f$ with respect to $\omega$ are exactly those in the set $1\times S^1\times 0$ which corresponds to $\gamma$. Thus one can also compute the local Hamiltonian Floer homology of this flow, which we call $\phi_{\omega, f}$, and produce a homology group associated with $\gamma$.\\

We say that the symplectic homology of $\gamma$ is the local Hamiltonian Floer homology of $\phi_{\omega, f}$ and the $S^1-$equivariant symplectic homology is the  $S^1-$equivariant local Hamiltonian Floer homology of $\phi_{\omega, f}$. The advantage of the latter is that it produces a homology which has exactly one generator per simple orbit, which is discussed in \cite{GG:gap}. This is helpful when one's goal is counting Reeb orbits via these homology groups, as is the case when proving results about the multiplicity of Reeb orbits, as is the case in \cite{AGKM,GM,GS,GG:convex}.\\

In this paper we prove two isomorphisms which show that the symplectic homology of $\gamma$ is entirely determined by the local Hamiltonian Floer homology of the return map of $\phi$. The first of the two isomorphisms is on the local symplectic homology of our isolated simple Reeb orbit, and the second is an isomorphism in $S^1$-equivariant symplectic homology of $\gamma$.

 \begin{Theorem}
\label{thm:iso1} Let $\gamma$ be a simple isolated closed Reeb orbit. Let $\phi$ be its return map. Then we have the following isomorphism in local homology for any coefficient ring $R$
\[
	SH_*(\gamma) \cong  HF^{loc}_*(\phi) \oplus  HF^{loc}_{*-1}(\phi).
\]
\end{Theorem}
The left hand side of this expression is the local symplectic homology of $\gamma$, and the right hand side is the local Hamiltonian Floer homology of the return map $\phi$. Recall that, as stated in the introduction, this is related to the a result by McLean where he proved an upper bound in rank for these homology groups in \cite{Mc}.  \\
 
The second isomorphism that we prove is
\begin{Theorem}
\label{thm:iso2}
Let $\gamma$ be a simple isolated closed Reeb orbit. Let $\phi$ be its return map. Then we have the following isomorphism in local homology for any coefficient ring $R$
\[
	SH_*^{S^1}(\gamma) \cong   HF^{loc}_*(\phi).
\]
\end{Theorem}
The left hand side of this expression is the $S^1$-equivariant symplectic homology of $\gamma$, and the right hand side is the local Hamiltonian Floer homology of $\phi$.\\*

It is worth noting that this isomorphism is independent of the choice of the cross product structure on $\Sigma=S^1\times B^{2n}$. For any bundle you take over $S^1$ to represent $\Sigma$ one can look at the return map of the fiber over some fixed $\theta_0$ in $S^1$; the local Hamiltonian Floer homology of this map will be isomorphic to that of $\phi$ up to a shift in degree. In addition one can define the symplectic homology of $\gamma$ independent of this choice of trivialization and still obtain something isomorphic, up to a shift in degree, to the symplectic homology we specify.\\

The role that the cross product plays here is that it allows us fix a frame for $B^{2n}$ along $S^1$, and thus allows us to determine an isotopy for the return map. This choice of frame will canonically extend to a frame in the setting of the symplectic homology of $\gamma$, which will determine the Conley Zehnder index and grading in this setting. Thus, when we state this isomorphism whatever comparison we make in degree is relative to the choice $\Sigma=S^1\times B^{2n}$ and how it subsequently determines the grading for our homology groups.\\

Finally, it is important that we mention that Theorem \ref{thm:iso2} is in some sense known, albeit not through a direct proof. Namely, in the case where the coefficient ring contains $\mathbb Q$, Theorem \ref{thm:iso2} can be deduced by applying two results regarding homology of contact manifolds. First, we know that the positive part of equivariant symplectic homology is isomorphic to the linearized or cylindrical contact homology, when it is defined. This is due to \cite{BO}, a paper by Bourgeois and Oancea where they prove this result in a global setting, when the coefficient ring contains $\mathbb Q$. This argument translates word for word to the local setting and when the orbit $\gamma$ is simple the contact homology is defined, see \cite{HM:contacthom}. As stated in the introduction, in \cite{HM:contacthom} they proved that the local contact homology of an isolated Reeb orbit $\gamma$ is isomorphic to the local Floer homology of the return map in Proposition 5.1. This together with  \cite{BO} proves the claim. Alternatively, one may use Lemma 5.2 of \cite{EKP} from Eliashberg, Kim, and Polterovich's paper which shows that contact homology is isomorphic to the Hamiltonian Floer homology of a certain Hamiltonian. When we look at these results in our local setting the Hamiltonian Floer homology we get from the \cite{EKP} result corresponds to a Hamiltonian which generates the germ of the return map $\phi$.\\

The advantage of the proof presented here is two-fold. First, it is a direct proof not involving contact homology whatsoever. As a result, it does not involve many of the additional complex tools which are employed in the above results. Secondly, our result remains true regardless of which coefficient ring one chooses to compute ones homology with. \\*

\subsection{What this result tells us about symplectically degenerate maxima}
\label{sec:SDM}

In this section we give two different definitions of a Reeb orbit being a symplectically degenerate maximum (SDM), and explain how our result shows their equivalence for simple Reeb orbits. This can be seen as an immediate corollary to Theorem \ref{thm:iso2}. The two definitions we give here are slightly different from the statements which involve iterations of an orbit having homology supported in certain degrees. That is, the definitions one usually sees in Conley conjecture type arguments. Nevertheless, the definitions we give here are equivalent to those involving the homology of iterates.\\

First, let's look at the symplectic homology definition of an SDM in which we are interested. Let $x$ be an isolated Reeb orbit in a contact manifold $(Y,\alpha)$. We say that $x$ is a \emph{symplectically degenerate maximum} if $SH^{S^1}_{\hat \mu(x) +n}(x)\neq 0$ where $\hat \mu$ is the mean index as defined in \cite{Lo,SZ,GG:convex} taken with respect to a trivialization of $(1-\delta,1+\delta) \times S^1\times B^{2n}$, and where $n$ here comes from the dimension of the contact manifold $Y$ which is $2n+1$. This definition is sometimes called being a contact SDM as given in \cite{GHHM}. This definition, as well as its iterated analogue, are discussed at some length in  \cite{GS}.\\

On the other hand, we are also interested in the Hamiltonian diffeomorphism which is determined by the germ of the return map $\phi$. If $\phi$ has an isolated fixed point $\hat x$, we say that $\hat x$ is an SDM if $HF^{loc}_{\Delta(\hat x)+n}(\phi) \neq 0$. This is the definition given in \cite{Gi:CC}. Here $\Delta(\hat x)$ is the mean index of $\hat x$ using the same definition as above, except in this case it is taken with respect to a trivialization on $B^{2n}$. This is discussed at greater length in \cite{GG:CCAB}. Now we may give the second definition for a Reeb orbit to be an SDM. An isolated Reeb orbit $x$ is a \emph{symplectically degenerate maximum} if its associated return map has an isolated fixed point $\hat x$ and this fixed point is an SDM in the Hamiltonian Floer homology sense which we just described.\\

Notice that $\hat \mu$ and $\Delta$ have the same definition from the perspective of linear algebra except that $\hat \mu$ is defined on an enlarged vector space. More specifically, if we take a trivialization of $B^{2n}$ it canonically extends to a trivialization of $(1-\delta,1+\delta) \times S^1\times B^{2n}$. If we take $\Phi^t \in Sp(2n)$ to be the path corresponding to the return map $\phi$ which determines $\Delta(\hat x)$. Then the mean index $\hat \mu(x)$ on $(1-\delta,1+\delta) \times S^1\times B^{2n}$ is determined by the path
\[
	\left( \begin{matrix}
	1 & a(t)\\
	0 &1
	\end{matrix} \right) \oplus \Phi^t \in Sp(2n+2).
\] From this we can see that in fact in the case of our Reeb orbit we have $\hat \mu( x)=\Delta(\hat x)$. \\

\begin{Corollary}
\label{coro:SDM}
Let $x$ be an isolated simple Reeb orbit. Then the two different definitions given above for $x$ to be a symplectically degenerate maximum are equivalent.
\end{Corollary}

\begin{proof}
This follows immediately from Theorem \ref{thm:iso2} together with the above stated fact that $\hat \mu( x)=\Delta(\hat x)$. 
\end{proof}

\subsection{Organization of the paper and outline of the proof}
\label{sec:org}

Here we lay out the organization of the paper and some guiding principles. We  also give a sketch of the result. \\

In Section \ref{sec:prelim} we give more complete definitions and pertinent results for the different homology groups involved. We also spell out the relevant formal properties of said homology groups. \\

After this we begin the proof of the main theorems.  The proof of Theorems \ref{thm:iso1} and \ref{thm:iso2} share much of the same set up and only diverge in approach at the very end of the argument, which we handle specifically in Sections \ref{sec:proof1} and \ref{sec:proof2} respectively. The majority of the work which applies to both approaches can be explained in rough terms as follows.\\

In broad strokes, the strategy is to apply the K{\"u}nneth formula for Hamiltonian Floer homology to the symplectic homology of $\gamma$. Much of the work done throughout the paper is done in order to put our local model in a form where we are able to apply the K{\"u}nneth formula. We may apply the  K{\"u}nneth formula for Hamiltonian Floer homology if we have two Hamiltonians defined on a product of symplectic manifolds. The homology is computed by looking at the time-one flow of the associated Hamiltonian vector fields. So it is a valid perspective to have two Hamiltonian vector fields defined on the product of symplectic manifolds.\\ 

In Section \ref{local model:HM} we give a detailed description of a local model for the symplectization of a contact manifold in the neighborhood of an isolated Reeb orbit. In the local model we have a Hamiltonian $f$, a symplectic form $\omega$ and the space we are interested in is $W=(1-\delta,1+\delta)\times S^1\times B^{2n}$ with coordinates $(r,\theta, z)$. Here the Hamiltonian $f$ only depends on the $r$ coordinate. The symplectic form $\omega$ is defined by taking $\omega=d(r\alpha)$. By a result from \cite{HM:contacthom} we know that there is some neighborhood $\Sigma$ of $\gamma$ on which we have $\alpha = Hd\theta + \lambda_0$ where the one form $\lambda_0 $ is a primitive of $\sigma$, $\sigma$ is the standard symplectic structure on $B^{2n}$, and $H$ is a periodic Hamiltonian on $B^{2n}$. In this symplectization the Reeb orbit $\gamma$ corresponds to the set $1\times S^1 \times 0$. The Hamiltonian vector field of $f$ with respect to $\omega$ is 
\[
	X_{\omega ,f} = \frac{f'(r)}{F}\left(\frac{\partial}{\partial\theta} -X_H\right).
 \] The Hamiltonian vector field $X_H$ is defined by the equation $-dH = \iota_{X_H}\sigma$, and $F=H-\lambda_0(X_H)$.\\

Our goal is to replace the Hamiltonian vector field for $f$ with respect to $\omega$ with one which decomposes as two Hamiltonian vector fields, one that only depends on the $r$ and $\theta$ coordinates, and one that depends only on the $z$ coordinates. We want to make this replacement in such a way that it does not affect the relevant homology groups up to isomorphism. Once this is done we will make sure that these Hamiltonian vector fields come from genuine Hamiltonians with respect to some symplectic structure on the model space $W$, and that the homology this computes is isomorphic to the original symplectic homology.\\
 
More specifically, we show the Hamiltonian vector field $X_{\omega ,f}$ computes the same local Floer homology as that computed by a Hamiltonian vector field which roughly resembles $f'(r)\frac{\partial}{\partial\theta} -X_H.$ The first vector field in this summand depends only on $r$ and $\theta$ and the second vector field depends only on the $z$ coordinates. There are a number of modifications that must be made in order to produce a Hamiltonian vector field of the desired form, and they must be done in such a way that homology is preserved.\\

The first issue addressed is the factor $F$ from $X_{\omega ,f}$; we handle the removal of this term by introducing a new local model in Section \ref{new model}. Secondly a priori the Hamiltonian $H$ is not autonomous as it lives in the phase space $\Sigma = S^1\times B^{2n}$ and is potentially $\theta$ dependent. Therefore, we need to replace $H$ by an autonomous counterpart $\hat H$. This is the content of Proposition \ref{autonomous H}, the proof of which is the most technically involved step we take in proving our main results. As for the side of the isomorphism involving the Hamiltonian Floer homology of the return map $\phi$, we will leverage the formal properties detailed in Sections \ref{FH:prop} and \ref{EFH:prop} together with well known techniques for modifying Hamiltonians on Euclidean space. Additionally, we need this replacement to be done in such a way that we are not affecting the symplectic homology groups up to isomorphism. The primary tools for doing this are developed in Lemma \ref{thm:lemma1} and Lemma \ref{thm:lemma2}, which give sufficient conditions for the types of modifications we may do to the Hamiltonian $H$ while still producing an isomorphism in symplectic homology. We find the desired autonomous $\hat H$, that is we give the proof of Proposition \ref{autonomous H}, in Section \ref{sec:propproof} as well as verifying that the any replacements made in the process satisfy all the technical conditions required to produce isomorphisms in the relevant local homology groups.\\

After these modifications we are left with a Hamiltonian vector field of the form $f'(r)\left(\frac{\partial}{\partial\theta} -X_{\hat H}\right).$ At this point there are two remaining issues preventing us from applying the K\"unneth formula. The first problem is that $f'(r)$ is being multiplied with the $X_{\hat H}$ term, and the second is that the contribution that $\hat H$ is making to the Hamiltonian vector field is contained inside of a symplectic form. We handle both of these issues at once by making a change of variable. The statement of Proposition \ref{autonomous H} and the change variable both occur in Section \ref{sec:prep for Kunneth}. \\
 
Finally, we are able to apply the K{\"u}nneth formula in the separate cases of symplectic homology and $S^1-$equivalent symplectic homology to finish the proof of the main Theorems in Sections \ref{sec:proof1} and \ref{sec:proof2}. \\

As a final note, some of the notation in the proofs which follow will take a slightly different form from that which appears in this sketch. As we lay out conventions and give specific notation for functions mentioned above, the final notation will be different. At each step, however, the expressions will have roughly the same form as those described above up to a sign or a change in labeling.\\

\section*{Acknowledgements} The author is grateful to Alberto Abbondandolo, Erman Cineli, Cheyenne Dowd, Umberto Hryniewicz, Leonardo Macarini, Marco Mazzucchelli, Mark McLean, and Otto van Koert for useful discussions and remarks. The author is particularly grateful to Viktor Ginzburg for his invaluable input and guidance throughout the project. Finally, the author is grateful to his wife for her support throughout the process. This paper was partially supported through a fellowship awarded by the math department at UC Santa Cruz.

\section{Background}
\label{sec:prelim}

In this section the aim is to give the relevant homology definitions and results as well as to set notation and conventions used throughout the paper. Specifically we give a very brief review of local Hamiltonian Floer homology, local symplectic homology, and their equivariant counterparts. In addition, we lay out relevant formal properties of these homology groups that will be used throughout the paper. The reader may consider consulting this section only as needed.\\

\subsection{Local Floer homology}
\label{sec:FH}

In this section we give a short recollection of the notion of local Floer homology. All of the definitions and results in this section are standard and well known. This homology was considered in a more general setting by Floer in \cite{F14,F15}; for a more specific account of local Floer homology the reader is referred to \cite{GG:gap,Mc}.\\

\subsubsection{Definition of local Floer homology:}
\label{FH:def}

Let $(M,\omega)$ be a symplectic manifold. Let $x$ be an isolated one-periodic orbit of a periodic Hamiltonian $H:S^1\times M \to \mathbb R$. Select a sufficiently small tubular neighborhood of $U$ of $x$ and then take a $C^2$-small perturbation $\tilde H$ of $H$ supported in $U$. More precisely, let $U$ be a neighborhood of $x$ in the extended phase space $M\times S^1$ and let $\tilde H$ be $C^2$-close to $H$, equal to $H$ outside of $U$, and such that all the one-periodic orbits of $H$ that enter $U$ are nondegenerate. With some abuse of notation we consider $U$ as both a neighborhood in $M$ and in $S^1\times M$.\\

Consider the one-periodic orbits of $\tilde H$ contained in $U$. Then every Floer trajectory $u$ connecting two orbits is also contained in $U$ if both $||\tilde H -H||_{C^2}$ and $\sup(\tilde H-H)$ are small enough. Indeed the energy of $u$ is equal to the difference of the action values of the orbits and is thus bounded above by $O(||\tilde H-H||_{C^2})$; from this it follows that $u$ takes values in $U$, see \cite{Sa1,Sa2}. Furthermore, all the transversality conditions are satisfied for moduli spaces of Floer anti-gradient trajectories of one-periodic orbits of $\tilde H$ contained in $U$.\\

In this setting the compactness theorem applies and we can construct a Floer complex for the one-periodic orbits of $\tilde H$ which are contained in $U$ in the standard way. This complex is independent of the choice of perturbation $\tilde H$ and of the almost complex structure used. The resulting homology group is denoted by  $HF^{loc}_*(p,\omega,H)$, where $p$ is the fixed point associated with $x$ and $\omega$ is the symplectic form on $U$. If the fixed point $p$ and the symplectic form  $\omega$ is understood, we may opt to use the more compact notation of simply $HF^{loc}_*(\phi^1_H)$. Note that the isotopy $\phi^t_H$ is essential to have the grading of $HF^{loc}_*(\phi^1_H)$ fixed and that isotopy is suppressed in the more compact notation. The grading of this complex is determined by Conley Zehnder index, this is described at length in \cite{Sa2}.\\

In general, $x$ in this construction can be replaced by an isolated connected compact set $F$ of one-periodic orbits of $H$.\\

\subsubsection{Properties of local Floer homology}
\label{FH:prop}

In this section we lay out the well known properties of local Floer homology which we will use throughout the paper. We provide references for each property which the reader may consult for a more detailed account.

\begin{itemize}

\item[(LFH1)]  Take a family of Hamiltonians $H^s$ and a family of symplectic forms $\omega_s$, $s\in [0,1]$, such that the set of fixed points $F$ of the time-one Hamiltonian flow of $H^s$ with respect to $\omega_s$, which we call $\phi^1_{H^s}$, is a  {\it uniformly isolated} fixed point set, that is for $s \in [0.1]$, the set $F$ is independent of $s$ and is the set of fixed points of $\phi^1_{H^s}$ for all $s$ in some neighborhood $U$ which contains $F$ and is independent of $s$. Then $HF_*^{loc}(F, \omega_s,H^s)$ remains constant, and as a result $HF_*^{loc}(F, \omega_0,H^0) \cong HF_*^{loc}(F, \omega_1,H^1).$ For a more detailed description see \cite{Gi:CC,Vi:GAFA}.

\end{itemize}

It is worth noting that  in (LFH1) we may consider either $\omega_s$ or $H^s$ to be constant and the result remains true; that is to say we may vary just the Hamiltonian, just the symplectic form, or vary both simultaneously. In our setting, we will typically be considering a situation in which we are only varying the symplectic form or one in which we are only varying the Hamiltonian.\\

\begin{itemize}

\item[(LFH2)] Let $\phi^t_H$ be an isotopy for a time dependent Hamiltonian $H_t$ and let $\psi^t=\phi^t_G$ be a loop of Hamiltonian diffeomorphisms defined on neighborhood of $p$ and fixing $p$, that is $\psi^t(p)=p$ for every $t\in S^1$. Finally, let $J_t$ be a rime dependent family of almost complex structures. Then 
\[
	HF^{loc}_*(p,\omega, G\#H, \psi^t_*J_t)\cong HF^{loc}_{*-2\mu}(p,\omega, H, J_t)
\] Where $\mu$ is the Maslov index of the loop $t\mapsto d\psi^t \in SP(T_pM)$, and $H$ is any Hamiltonian with $p$ as a fixed point of its time-one map. Note that this is in contrast to the case where the Hamiltonian $G$ is globally defined and there is no shift in the degree of homology. This property shows us that local Floer homology is determined up to isomorphism, and possibly a shift on degree, by the time-one map of a Hamiltonian diffeomorphism. 

\item[(LFH3)] {\it K{\"u}nneth formula for Floer homology}. Let $(M_1,\omega_1)$ and $(M_2, \omega_2)$ be symplectic manifolds with Hamiltonians $H_1$ and $H_2$ respectively. Then 
\[
	HF_*(\phi^1_{H_1\oplus H_2}) = \bigoplus_{j+k=*}HF_j(\phi^1_{H_1})\otimes HF_k(\phi^1_{H_2}).
\] Where $H_1\oplus H_2$ is the natural Hamiltonian defined on the product $M_1\times M_2$.
\end{itemize} 

There is also a K{\'u}nneth formula for local Floer homology. Let $\gamma_1$ and $\gamma_2$ be one-periodic of Hamiltonians $H_1$ and $H_2$, with fixed points $p_1$and $p_2$, defined on local neighborhoods $(U_1,\omega_1)$ and $(U_2,\omega_2)$ respectively . Then $HF_*((p_1,p_2),\omega_1\oplus\omega_2,H_1\oplus H_2) = \bigoplus_{j+k=*}HF_j(p_1,\omega_1,H_1)\otimes HF_k(p_2,\omega_2,H_2)$. We will have occasion to apply both versions in the argument which follows and will simply refer to this as the K{\"u}nneth formula. \\

Additionally, this formula becomes a spectral sequence when the coefficient ring, denoted as $R$, which the homology is computed over is not a field. The pages of this exact sequence are comprised of the higher Tor$^R$ groups of $HF_*(p_1,\omega_1,H_1), HF_*(p_2,\omega_2,H_2)$. However, if either of these Floer homology groups are finitely generated free $R$-modules in every degree, as will be the case in this paper, all of the higher Tor groups will all vanish and the spectral sequence will collapse to the isomorphism written above. In our setting, we are using specifically the K{\"u}nneth formula for compact symplectic manifolds with contact type boundaries; one can consult \cite{Oa} for a more in depth description of this result.\\

\subsection{Local $S^1$-equivariant Floer homology}
\label{sec:EFH}

In this section we examine the a local version of $S^1$-equivariant Floer homology. In particular, we are interested in the equivariant local Floer homology of an isolated orbit, which will be the basis for the definition of the local $S^1$-equivariant symplectic homology of an isolated Reeb orbit.\\

Let $x$ be an isolated one-periodic orbit of an autonomous Hamiltonian $H$ defined on a neighborhood $U$ of $x$. Now take a natural extension of $H$ to the space $S^1\times U\times S^{2k+1}$ with coordinates $(\beta, z,\xi)$  where $H(\beta, z,\xi) = H(z)$. Then take a $C^2$-small perturbation of $H$, call it $\tilde H$, which is invariant under the following $S^1$-action $\tilde H(\beta+\alpha ,z, \alpha\cdot\xi)= \tilde H(\beta, z,\xi)$, where $\alpha \in S^1$ and the dot denote the Hopf action of $\alpha$ on $\xi$. Under this perturbation of $H$ to $\tilde H$ the orbit $x$ breaks down into a collection of transversely nondegenerate critical $S^1$-orbits $y_i$ lying close to $C= Cx\times_{S^1}S^{2k+1}$, where $Cx$ is the orbit of $x$ in the space of loops on the ambient manifold. Let $CF^{S^1}_*(x)$ be the relatively graded vector space or free module generated by the $y_i$. Just as in the case of local Floer homology as above and the Floer trajectories $\tilde u$ connecting some $y_j$ to some $y_i$ stay close to $C$. A proof of this can be seen in \cite{GG:convex}. As a result a differential can be defined on $CF^{S^1}_*(x)$ in the standard way for $S^1$-equivariant Floer homology groups, that is counting certain special Floer trajectories connecting $y_i$ to other $y_j$ modulo both the $S^1$ action and the natural $\mathbb R$ action on the Floer trajectories. We can take the $S^1$-equivariant homology of this complex for any sphere of dimension $2k+1$.  Finally, we take a directed limit on the dimension of the sphere (that is $2k+1$ as $k$ tends to infinity). We call the resulting homology groups the local $S^1$ equivariant Floer homology of $x$ and denote it by $HF^{S^1}_*(x)$. We may also denote this as $HF^{S^1}_*(x,\omega,H)$ when we want to keep track of both the symplectic form $\omega$ and the Hamiltonian $H$; the resulting homology is independent of perturbation $\tilde H$ taken for $H$. As in the case with the local Floer homology above, this construction can be extended from a single isolated orbit to a compact, connected, isolated set of one periodic orbits which we will also call $F$ in this case.\\

\subsubsection{Properties of local $S^1$-equivariant Floer homology}
\label{EFH:prop}

The local $S^1$-equivariant Floer homology enjoys several formal properties similar to those of local Floer homology. The one which is most relevant to the proof of our results is laid out here. The construction of continuation maps for the local equivariant Floer homology extends basically word for word to the equivalent version of statement (LFH1) for local Floer homology. It is only restated in its entirety here so the statement is appropriately adapted to the equivariant setting and for ease of reference.\\

\begin{itemize}

\item{(LEFH1)} Consider a family of autonomous Hamiltonians $H^s$ and a family of symplectic forms $\omega_s$, $s\in [0,1]$, such that the set of fixed points $F$ of $\phi^1_{H^s}$ is a uniformly isolated fixed point set for all $s \in [0,1[$, that is the set $F$ consists only of the fixed points of $\phi^1_{H^s}$ in some neighborhood $U$ independent of $s$. Then $HF_*^{S^1}(F, \omega_s,H^s)$ remains constant, and as a result $HF_*^{S^1}(F, \omega_0,H^0) \cong HF_*^{S^!}(F, \omega_1,H^1).$ 

\end{itemize}

Again, similar to the non equivariant setting, we may consider either $\omega_s$ or $H^s$ to be constant and the result remains true. That is to say we may vary just the Hamiltonian, just the symplectic form, or vary both simultaneously.\\

With the relevant types of local Floer homology laid out we go on to define local symplectic homology.\\

\subsection{Local symplectic homology}
\label{sec:SH}

This section gives a brief overview of the definitions of two types of symplectic homology of an isolated one-periodic simple Reeb orbit $\gamma$. These definitions originally appear in \cite{GG:convex}. The first, which we will call the symplectic homology of $\gamma$, is built using local Floer homology. The second is the $S^1$-equivariant symplectic homology of the orbit $\gamma$, and is defined in terms of $S^1$-equivalent local Floer homology. We will often refer to the latter as the equivariant symplectic homology of $\gamma$. 

\subsubsection{Definitions of symplectic homology groups of a Reeb orbit}
\label{SH:def}

Let $\gamma$ be an isolated simple period-one Reeb orbit, let $\Sigma=S^1\times B^{2n}$ be a neighborhood of $\gamma$ in which there are no other Reeb orbits of period $t$ for $t\in(1-\eta, 1+\eta)$, and the return map $\phi$ from $0\times B^{2n}$ to itself is Hamiltonian. Let $\alpha$ be the contact form restricted to $\Sigma$. Then we may take the symplectization $W=(1-\delta,1+\delta) \times \Sigma$ with coordinates $(r, \theta, z)$ and the symplectic form $\omega = d(r\alpha)$. Next we take a Hamiltonian $f(r)$ with $f(1)=f'(1)=1$ and $f''>0$, then the Hamiltonian flow $\phi^1_{\omega, f}$ has $1\times S^1 \times 0$ as an isolated fixed point set. This set corresponds to $\gamma$ in the symplectization. We then define the {\bf local symplectic homology} of $\gamma$ to be 
 \[
SH_* (\gamma)= HF^{loc}_*(\gamma=1\times S^1\times 0,\omega ,f).
\]
As an aside, this definition comes as a slight adaptation of the one presented in \cite{GG:convex}. In that paper, they are considering local symplectic homology in a context where there is a global symplectic homology which is defined and important to their proof. The conditions which define local symplectic homology in that paper are defined relative to a global context and are determined by taking an admissible Hamiltonian from a cofinal sequence of admissible Hamiltonians which are convex and linear outside a compact set with slopes trending to infinity. There is also a condition that the Hamiltonians in this context are nonpositive on a certain compact set. This is done in order to separate the fixed points that lie the symplectic filling of the contact manifold which are guaranteed by Arnold conjecture. In that setting, they compute local symplectic homology by taking a Hamiltonian whose linear slope outside of a compact set is larger than the period of the Reeb orbit whose symplectic homology they would like to compute. Then they restrict the Hamiltonian to an appropriate neighborhood of the Reeb orbit within the symplectization of a symplectic manifold with boundary of contact type.\\

It is due to these facts that the statements here appear different than in \cite{GG:convex}, as we have no global context in which we are interested. Thus we may assume our simple Reeb orbit has period-one and ignore the conditions that are required of the Hamiltonian outside the small neighborhood of the Reeb orbit. If one translates their conditions to a purely local setting, it is fairly straightforward to see the conditions outlined here are equivalent to those in \cite{GG:convex}. The symplectic homology of $\gamma$ is independent of any choice of $f$ provided it satisfies the conditions we lay out above, namely that $f(1)=f'(1)=1$ and $f''>0$.\\

From this definition it is clear what the generators of our complex are and the differential. However, missing from our discussion up to this point is how the grading is defined. The grading for our local symplectic homology is determined by the trivialization taken for our tubular neighborhood $\Sigma=S^1\times B^{2n}$ of $\gamma$. When we pick a cross product structure $\Sigma$ we determine a frame for $\Sigma$ which naturally extends to $W$. This allows us to compute the Conley Zehnder index of any fixed point and thus determine the grading for our complex.\\

When we define the $S^1$-equivariant symplectic homology of $\gamma$, we use a symplectization $(W,\omega)$ and a Hamiltonian $f(r)$ of the same type as above. We define the {\bf $S^1-$equivariant symplectic homology} of $\gamma$ to be 
\[
	SH_*^{S^1}(\gamma) = HF^{S^1}_*(\gamma=1\times S^1\times 0,\omega, f).
\] The grading here is also determined by the trivialization $\Sigma =S^1 \times B^{2n}$ in a manner analogous to that described above.\\

To see why we may assume $\gamma$ has period-one, as well as a specific form that $\alpha$ may take in order to make computations in local symplectic homology, see Section \ref{local model:HM}.\\
	
\subsubsection{Properties of symplectic homology}
\label{SH:prop}

The formal properties of symplectic homology we use are all inherited from the respective formal properties of the associated Floer homology. We will need a few additional tools for the purpose of our proof and we build those in Section \ref{sec:background}.\\

\section{Preliminary results}
\label{sec:background}

In this section the aim is to introduce a local model for a neighborhood of a Reeb orbit which will makes local homology computations easier as well as well as putting us one step closer to applying the K\"unneth formula. After this, we introduce two tools which enable computations on this local model. In order to produce the new local model we start with a local model of a simple Reeb orbit introduced by Hryniewicz and Macarini in \cite{HM:contacthom}. We do this in Section \ref{local model:HM}. In this section we also prove some needed facts about the associated flows and Hamiltonians functions involved. We will then apply these facts in Section \ref{new model} as they will allow us to show that the new local model computes the same homology as the original local model from \cite{HM:contacthom}. We follow with the proof of Lemma \ref{thm:lemma1} and Lemma \ref{thm:lemma2}, which will tell us what kind of replacements we can make in terms of the return map $\phi$ in our new local model without affecting the symplectic homology of $\gamma$ up to isomorphism.\\

\subsection{The setting and the local model}
\label{local model:HM}

Before we give a thorough description of a model neighborhood of a Reeb orbit $\gamma$ where all our homology computations will take place, let us first recall some relevant information. We have a contact manifold $(Y,\alpha)$ of dimension $2n+1$. We have a Reeb vector field $R$ with associated Reeb flow $\psi^t$. We have a period-$T$ closed simple Reeb orbit $\gamma$. The orbit $\gamma$ has a neighborhood $\Sigma \subset Y$ with trivialization $S^1\times B^{2n}$ where the return map $\phi$ defined by following Reeb flow from the slice $0\times B^{2n}$ to itself is Hamiltonian.\\

Having recalled the above facts, we would like to make local homology computation along $\gamma$. In order to do this we will need a local model for what a contact manifold looks like along a Reeb orbit. For this we look to the following result from \cite{HM:contacthom}.
\begin{Lemma} 
\label{HM:model}
From Lemma 5.2 in \cite{HM:contacthom}. Let $\gamma$ be a simple closed period$-T$ Reeb orbit of a contact manifold ($Y, \alpha)$ of dimension $2n+1$.Then there exists a tubular neighborhood $\Sigma = S^1\times B^{2n}$ of $\gamma(\mathbb R)$, where $B^{2n}\subset \mathbb R^{2n}$ is a small ball centered at the origin, with coordinates $(\theta, q_1,...,q_n, p_1,....,p_n)$, such that $\gamma(\mathbb R)= S^1\times 0$, $\alpha= Hd\theta +\lambda_0$ where $H:\Sigma \to \mathbb R$ satisfies $H_\theta(0) =T$, $dH_\theta(0)=0$, and $\lambda_0 = \frac{1}{2} \Sigma_1^n q_i dp_i -p_i dq_i$.
\end{Lemma}

Now that we have a specific local model we can begin doing computations on the symplectic homology of $\gamma$, as well obtaining an explicit form for the return map $\phi$ in terms of the Hamiltonian $H$. In order to do this we first look at the Reeb vector field. We then give an explicit description of the return map, and we finally go on to give the symplectization of this model. We prove that for an appropriate choice of symplectization and Hamiltonian we can satisfy the required conditions in the definition of symplectic homology. Namely, that the set corresponding to $\gamma$ is an isolated fixed point set for the time-one Hamiltonian flow on the symplectization. This will show that the local model is a suitable place to begin the proof of our isomorphisms. Finally we will show some technical conditions which we will need in the next step of our proof, Section \ref{new model}.\\

The first thing we do in our local model is make a time change. If we consider $\frac 1 T \alpha$ rather than $\alpha$, we have that the Reeb vector field of  $\frac 1 T \alpha$ is $TR$ where $R$ is the Reeb vector field of $\alpha$. The time-one flow of this new vector field is equal to the time-$T$ flow of the original vector field. This permits us to study the time-$T$ flow of the Reeb vector field of $\alpha$ by looking at the time-one flow of the Reeb vector field of $\frac 1 T\alpha$. Doing this allows us to avoid carrying along the constant $T$ in the many computations that follow as well as letting us adhere to the convention of studying time-one flow when we look at Hamiltonian Floer homology. Thus from this point forward we will assume without loss of generality that $\gamma$ has period-one with respect to $\alpha$ and due to Lemma \ref{HM:model} we have that $H_\theta(0) =1$.\\

Now we look at the Reeb vector field of our local model. Recall that $\alpha= Hd\theta +\lambda_0$ and that $d\alpha= dH\wedge d\theta +\sigma$ where $\sigma=d\lambda_0$ is the standard symplectic structure on $B^{2n}$. As one may easily verify, the Reeb vector field of $\alpha$ is
\begin{align}\label{RVF}
	R= \frac{1}{F}\left(\frac{\partial}{\partial\theta} -X_H\right),
\end{align}
where $\iota_{X_H}\sigma = -dH$ and $F= H-\lambda_0(X_H)$, noting that $F$ is nonzero at least on a small neighborhood of $\gamma$. In particular, we may consider this neighborhood to be $\Sigma$ if the radius of our $B^{2n}$ is taken to be small enough.\\
 
Now we look at the return map $\phi$. The periodic function $H$ defined as a part of the contact form $\alpha$ is the Hamiltonian which determines the return map in the following way. Note that $H$ is a one-periodic Hamiltonian on $B^{2n}$. The Hamiltonian flow of $-H$ in the extended phase space $S^1\times B^{2n}$ is given by computing the flow of $\frac{\partial}{\partial\theta} -X_H$. Then it is readily seen that this vector field and $R$ have the same integral curves, as we can obtain one from the other by multiplying by a nonzero function on $\Sigma$. Thus the time-one flow of $-X_H$, denoted $\phi^1_{-H}$, is the return map of following the Reeb flow from the slice $0 \times B^{2n}$ back to itself. Thus in this local model we have by definition that $\phi=\phi^1_{-H}$.\\

Now we take a symplectization of our model neighborhood which has the form $W =(1-\delta,1+\delta)\times \Sigma = (1-\delta,1+\delta)\times S^1 \times B^{2n}$ with coordinates $(r,\theta ,z)$. With some abuse of notation we will refer to both the sets $S^1\times 0\subset \Sigma$ and $1 \times S^1 \times 0 \subset W$ as $\gamma$ depending on the context. We take our symplectic form to be $\omega = d(r\alpha)$. Then for any Hamiltonian $f(r)$ that depends only on the $r$ coordinate the Hamiltonian vector field of $f$ with respect to $\omega$ is

\begin{align}\label{HVFomega}
	X_{\omega, f}=f'(r)R= \frac{f'}{F}\left(\frac{\partial}{\partial\theta} -X_H\right).
\end{align}

Note that if $f'(1) =1$, then this Hamiltonian vector field  $X_{\omega, f}$ agrees with the Reeb vector field on the level $r=1$, which corresponds to the contact manifold. Furthermore, if  $f'(1) =1$, the time-one fixed points of this vector field agree on that slice with the time-one fixed points of $R$.\\

It is at this point we recall that $\gamma$ is isolated in our setting, so we may choose $\Sigma$ such that it contains no closed Reeb orbits of period $t, t\in (1-\eta, 1+\eta)$ for some $\eta>0$. In addition, from this point forward we note that one can ensure that $f'(c) \in (1-\eta,1+\eta)$ for $c \in (1-\delta,1+\delta)$ which can be achieved for any appropriate choice of $f$ by simply shrinking $\delta$. We will need this condition as well as the convexity in order to ensure that passing to the symplectic setting doesn't introduce any time-one fixed points other than $\gamma$ as we will elaborate on in Lemma \ref{thm:one_periodic_are_reeb} below.\\

Drawing our attention to the time-one fixed points of $X_{\omega, f}$, it is the case that the time-one flow of $X_{\omega, f}$ has every point in $1\times S^1\times 0$ as a fixed point since it agrees with the Reeb vector field on the level $r=1$ and the time-one flow of $R$ has the points on $\gamma$ as fixed points. However, in order to satisfy the definition of symplectic homology we require that those are the only time-one fixed points of $X_{\omega, f}$.  From the definition of an isolated orbit we can prove the following Lemma connecting the dynamics of the model neighborhood $\Sigma$ to the dynamics of its symplectization.

\begin{Lemma}
\label{thm:one_periodic_are_reeb}
	Let $W$, $f$, and $\omega$ be as described above. A point $p \in W$ is a fixed point of the time-one flow of $X_{\omega, f}$, which we denote $\phi_{\omega, f}$, 	if 	and only if it lies on the Reeb orbit $\gamma = 1\times S^1\times 0$. That is $\gamma$ is an isolated fixed point set of $\phi^1_{\omega,f}$.
\end{Lemma}

\begin{proof}
	As we stated above, the points in $\gamma$ can be seen to be fixed points of $\phi_{\omega, f}$. So we need only show that if $p$ is a fixed point of $\phi_{\omega, f}$ then $p \in \gamma$. Let $\psi^t$ be the time-t flow of $R$. First we prove that $\gamma$ is an isolated Reeb orbit, that is there are no other Reeb orbits of period $\tau \in(1-\eta, 1+\eta)$ in the neighborhood $\Sigma$, if and only if $\psi^t$ has no fixed points other than $\gamma$ for $t\in(1-\eta, 1+\eta)$.\\
	
	We prove this by showing that $\psi$ has a fixed point of time-$\tau$ if and only if it has a Reeb orbit of period $\tau$. First, if $\psi$ had a Reeb orbit  $l(t)$ of period $\tau$, then for any point $p_0$ on that Reeb orbit, $p_0$ is a period-$\tau$ fixed point of the Reeb flow. This follows from the definition of being a period-$\tau$ closed Reeb orbit. Let $l(t_0)=p$;  then if $l$ is a closed period $\tau$ Reeb orbit it follows that $\psi^{t_0+\tau}(p) = l(t_0+\tau) =l(t_0) =\psi^{t_0}(p)$, where $p$ is an initial condition for our Reeb orbit. From this we see that $\psi^{\tau}(\psi^{t_0}(p))=\psi^{t_0}(p)$ showing us that $\psi^{t_0}(p)=p_0$, is a time-$\tau$ fixed point. On the other hand, if there is some $p$ such that $\psi^{\tau}(p) =p$, then for some $p_0= \psi^{t_0}(p)$, we have that $\psi^{\tau}(p_0)=\psi^{\tau}(\psi^{t_0}(p))=\psi^{t_0+\tau}(p)=\psi^{t_0}(\psi^{\tau}(p))=\psi^{t_0}(p)=p_o$, showing that $\phi^t(p) =l(t)$ is a closed period-$\tau$ Reeb orbit.\\
	
	Thus, we now know that the points on $\gamma$ are the only fixed points for the Reeb flow of time $t\in(1-\eta ,1+\eta)$. From this, we can see that $\phi_{\omega, f}$ only has time-one fixed points on $\gamma$, due to the fact that on the level $r=c$, we have that $X_{\omega, f}=f'(c)R$ and thus $ \phi_{\omega, f}(c,\theta, z)=(c,\psi^{f'(c)}(\theta, z))$. By our convexity assumptions for $f$ together with $f'(1)=1$ we know that $f'$ is only equal to 1 on the level $r=1$. In addition, we have $f'(c)\in (1-\eta, 1+\eta)$ as stated above, and as such $\psi^{f'(c)}(\theta, z)=(\theta, z)$ if and only if $c=1$ and $z=0$; that is if and only if $(c,\theta, z)\in \gamma$. Thus, the fixed points of $\phi_{\omega, f}$ are exactly those on $\gamma$.
\end{proof}

\begin{Remark}
\label{H fixes origin}
	As a consequence of the proof of the above Lemma, we can see that the time-t Hamiltonian flow of $-H$ on $B^{2n}$ has the origin as an isolated fixed point for $t\in (1-	\eta, 1+\eta)$, $\eta>0$, and has the origin as a fixed point for all time. Additionally, this is also true for the return map on any slice $\theta \times B^{2n} \subset \Sigma$.
\end{Remark}

The reason we are interested in $-H$ is because $\phi^1_{-H}$is the return map of the Reeb flow. Recall, the Hamiltonian vector field of $-H$ in the extended phase space is $\frac{\partial}{\partial\theta}-X_H$; as such both the Reeb vector field as well as the Hamiltonian vector field of $f$ with respect to $\omega$ on the level $r=c$ are both smooth functions times the vector field $\frac{\partial}{\partial\theta}-X_H$. Thus, these three flows have the same integral curves. When we say the Hamiltonian flow $f$ has the same integral curves we are referring to restricting the Hamiltonian flow of $f$ to the level $c\times S^1 \times B^{2n}$.\\

In particular, there is some time $T_0$ such that for any time $t\in(T_0-\eta, T_0+\eta)$ we have that the time-$t$ flow of $-X_H$, which we write $\phi^t_{-H}$, has the origin as an isolated fixed point. We have more information here, though, which is that the $\frac 1 F$ term which scales the Reeb vector field is equal to 1 along $\gamma \subset \Sigma$. From this it follows that along $\gamma$  the time-one flow of $-H$ in the extend phase space is equal to the Reeb flow along $\gamma$. Thus we know that $\phi_{-H}^t$ has the origin as an isolated fixed point for $t\in(1-\eta, 1+\eta)$. Furthermore, we know that the integral curves of $R$ agreeing with those of the Hamiltonian flow of $-H$ gives us that the origin in $B^{2n}$ is a fixed point of $\phi^{\theta}_{-H}$ for all $\theta$ though it may not be isolated. As a final comment, we can look at the return map for any slice $\theta_0\times B^{2n}$. Here, the return map is $\phi_{-H}^{1+\theta_0}\circ(\phi_{-H}^{\theta_0})^{-1}(z)$ and through the same logic we can deduce that $\phi_{-H}^{t+\theta_0}\circ(\phi_{-H}^{\theta_0})^{-1}(z)$ has the origin as an isolated fixed point for $t\in (1-\eta, 1+\eta)$.We will use these properties when we introduce the new local model in the following section.\\

\subsection{New local model }
\label{new local model }

Here we introduce a new local model  for the symplectization of our neighborhood of a Reeb orbit.

\begin{Proposition}
\label{new model}
Let $W$, $\omega, f$ be as they are defined in the assumptions of Lemma \ref{thm:one_periodic_are_reeb}. Let $H$ be the Hamiltonian discussed in Section \ref{local model:HM}. We introduce $\omega_H= d(r+H)\wedge d\theta+\sigma$ as a new symplectic form for the model space $W$. Then the Hamiltonian Floer homology of $f$ with respect to $\omega$ is isomorphic to the Hamiltonian Floer homology of $f$ with respect to $\omega_H$. From the definition of symplectic homology this implies the following two isomorphisms:
\[
	SH_*(\gamma) \cong HF^{loc}_*(1\times S^1\times 0,\omega_H,f) \qquad \text{and}\qquad 	SH_*^{S^1}(\gamma) \cong HF^{S^1}_*(1\times S^1\times 0,\omega_H,f). 
\] 
 
\end{Proposition}

\begin{proof}
First we show that the fixed point set of the time-one map of the Hamiltonian flow of $f$ with respect to $\omega_H$ is $\gamma=1\times S^1\times 0$. We denote this Hamiltonian flow as $\phi^1_{\omega_H, f}$. In this model the Hamiltonian vector field of $f$ is $X_{\omega_H, f}= f'(r)(\frac{\partial}{\partial \theta} -X_H)$. This is an advantage of the new model as we can compute the Hamiltonian flow of $f$ in terms of the Hamiltonian flow of $-H$ on $B^{2n}$. The time-$t$ flow of $X_{\omega_H ,f}$ is 
\[
	(r,\theta ,z) \mapsto (r, \theta+ f'(r)t, \phi_{-H}^{f'(r)t+\theta}\circ(\phi_{-H}^\theta)^{-1}(z)).
\] So, if we are looking for time-one fixed points, first we need to see when  $\theta+ f'(r)=\theta$. This only happens when $f'(r)\in\mathbb Z$. By definition, $f'(1)=1$, and $f'' >0$. As such this only happens when $r=1$. Thus any fixed point must be in $1\times S^1\times B^{2n}$. Now, drawing our attention to the last coordinate, we are looking for fixed points of $ \phi_{-H}^{1+\theta}\circ(\phi_{-H}^\theta)^{-1}$. By the discussion in Remark \ref{H fixes origin} we know that $0\in B^{2n}$ is the only fixed point of this map. Thus the fixed point set of $\phi_{\omega_H, f}^1$ is $\gamma= 1\times S^1\times 0$.\\

Next let us recall our original symplectic form $\omega =d(r\alpha)$ where $\alpha = Hd\theta + \lambda_0$ and $d\lambda_0 =\sigma$ is the standard symplectic form on $B^{2n}$. Then observe the following, $\omega|_\Sigma = \omega_H|_\Sigma$ and $\frac 1 c\omega|_{\Sigma_c} = \omega_H|_{\Sigma_c}$, where $\Sigma$ is the level $r=1$ and $\Sigma_C$ is the level $r=c$. Furthermore, note that in both models the return map for every level is $\phi^1_{-H}$.\\

Now we produce an isotopy between the symplectic forms which leave $\gamma$ uniformly isolated. Let $\omega_0=\omega_H$ and let $\omega_1=\omega$, then we take $\omega_s=(1-s)\omega_0+s\omega_1, s\in [0,1]$. We claim that $\omega_s$ is symplectic near $\gamma=1 \times S^1\times 0$. First, as $\omega_0|_\Sigma =\omega_1|_\Sigma$, we have $\omega_s|_\Sigma= dH \wedge d\theta+\sigma$ which is a maximally nondegenerate on $\Sigma$. Thus in order to show $\omega_s$ is symplectic near $\gamma$ we need only show that $\iota_{\frac{\partial}{\partial r}}\omega_s \neq 0$. Here we have $\iota_{\frac{\partial}{\partial r}}\omega_s = (1-s)d\theta +sHd\theta+s\lambda_0 =(1+s(H-1))d\theta +s\lambda_0$ as $H$ does not depend on $r$. Now at $0\in B^{2n}$ we have $H=1$ for all $\theta$ and $\lambda_0 =0$, which means that along $\gamma$ we have $\iota_{\frac{\partial}{\partial r}}\omega_s =d\theta \neq 0$. Thus $\omega_s$ is symplectic near $\gamma$ for all $s\in [0,1].$\\

Next we show that $\gamma$ is a uniformly isolated fixed point set for $\phi^1_{\omega_s,f}$. With $\Sigma_c$ as above, we recall that $c\omega_0|_{\Sigma_c}= \omega_1|_{\Sigma_c}$ and thus $\omega_s|_{\Sigma_c} = ((1-s)c+s)\omega_1|_{\Sigma_c}$. This implies that the Hamiltonian flow of $f$ with respect to $\omega_s$ has the same unparameterized characteristics on each $\Sigma_c$ for every $s\in [0,1]$. Thus the flow of $\phi^1_{\omega_s, f}$ for any initial condition is just a time change of $\phi^1_{\omega_1,f}$. The only thing we need to verify is that whatever time change we make is small enough that our fixed point set remains uniformly isolated. In particular we know that from the discussion in Remark \ref{H fixes origin} that for each $\theta_0\in S^1$ we have that the origin in $B^{2n}$ is a fixed point of $\phi_{-H}^{t+\theta}\circ(\phi_{-H}^\theta)^{-1}$ for all time $t\in (1-\eta, 1+\eta)$. Now the $\eta$ here may depend on $\theta_0$ but since $S^1$ is compact there is some $\eta$ such that this is true for all $\theta_0 \in S^1$. Now lets look at the Hamiltonian vector field for $f$ with respect to $\omega_s$,
\[
	X_{\omega_s, f} = (1-s)f'(r)+s\frac{f'(r)}{F}\left(\frac{\partial}{\partial\theta}-X_H\right).
\] If we are able to show that we can make $\left|1-\left((1-s)f'(r)+s\frac{f'(r)}{F}\right)\right|<\eta$ then we will be able to use the fact that $\phi_{-H}^{t+\theta}\circ(\phi_{-H}^\theta)^{-1}$ has an isolated fixed point at the origin for $t \in (1-\eta, 1+\eta)$ in order to show that $\phi_{\omega_s, f}$ has $\gamma$ as an isolated fixed point set. To see this, first note $f'(r)$ is equal to 1 when $r=1$. Secondly $F$ is a smooth function of $\theta$ and $z$. We know that when $z=0$ that $F=1$. Now looking at $(1-s)f'(r)+s\frac{f'(r)}{F}$, we can see that it is a smooth function of $s, r,\theta$ and $z$ and that along $\gamma$ it is equal to $1$ for all $s\in [0,1]$. Then using the fact that $[0,1]$ is compact, together with the fact this function is smooth, one can find a neighborhood of $\gamma$ which is independent of $s$ such that for all $s\in [0,1]$ we have $(1-s)f'(r)+s\frac{f'(r)}{F} \in (1-\eta, 1+\eta)$. We will continue to call this new neighborhood $W$, although it may require us to shrink $\delta$ or the radius of $B^{2n}$.\\

Now we finish the proof. For some fixed $r=c$, we know that each return map has $0\in B^{2n}$ as an isolated fixed point for time $t\in (1-\eta, 1+\eta)$. We know that the time change for any unparameterized curve lies in $(1-\eta, 1+\eta)$ for all $s\in [0,1]$. As such, for the $z$ coordinate the only possible fixed point is the origin. Thus on the level $r=c$ our fixed point set must lie in $c\times S^1\times 0$. On this set $F=1$ so the time change for our vector field along the $\theta$ coordinate becomes $f'(c)$. But by the convexity of $f$ this is only equal to one when $r=1$. So the only possible fixed points of $\phi^1_{\omega_s, f}$ lie in $\gamma=1\times S^1\times 0$. Finally, since $\omega_s|_\Sigma=\omega_1|_\Sigma$ we know that $\gamma$ is a fixed point set for all $s$. This shows that $\gamma$ is a uniformly isolated fixed point set for $\phi^1_{\omega_s, f}$ and hence by applying (LFH1), as well as (LEFH1) we know that
\[
	 HF^{loc}_*(\gamma,\omega, f) \cong HF^{loc}_*(\gamma,\omega_H,f)\qquad \text{and}\qquad  HF^{S^1}_*(\gamma,\omega, f)\cong HF^{S^1}_*(\gamma,\omega_H,f)\
\] completing the proof.

\end{proof}

We have shown that our new local model computes the same symplectic homology as the original. Next we fix notation which we will use for the remainder of the paper and go on to discuss the advantages of this new local model.\\

We first make the change of notation. The expression $\phi^1_{-H}$ is somewhat cumbersome. As such, from this point forward we will define $H$ as $-H$ instead. The effects of this change of notation are as follows. First $\alpha = -Hd\theta + \lambda_0$ and $\omega_H = d(r-H)\wedge d\theta + \sigma$, the return map of the flow is $\phi^1_{H}$, and finally the Hamiltonian vector field of $f$ with respect to $\omega_H$ is now $f'(r)\left(\frac{\partial}{\partial\theta} +X_H \right)$.\\

Introducing this local model was the first step in the proof of our theorems. The model not only makes computation simpler, but it is also removed the requirements imposed on the Hamiltonian $H$ in Lemma \ref{HM:model}. In the original model, it is required that $H$ is a nonzero constant along $\gamma$ in order for the differential form $d(r\alpha)$ to be a symplectic form in a neighborhood of $\gamma$. In our new local model, $\omega_H$ is always symplectic, independent of $H$, which will allow us to make more types of modification to the Hamiltonian $H$.\\

Our next goal is to replace the given Hamiltonian $H$ by something autonomous. Recall this is required in order to ultimately apply the K{\"u}nneth formula. Before we do this we develop two tools which enable us to replace the Hamiltonian $H$ with a new Hamiltonian, while preserving all relevant homology up to isomorphism. The rest of the section is devoted to the construction of these tools.\\

\subsection{Modifications which give isomorphisms in local homology}
\label{sec:modification}

In this section we prove two lemmas showing that we are allowed to apply certain modifications to a Hamiltonian $H$ on $B^{2n}$ while still preserving the symplectic homology groups up to isomorphism. More specifically, both of our local symplectic homology groups are isomorphic to some Floer homology of a Hamiltonian $f$ on $(1-\delta,1+\delta) \times S^1\times B^{2n}$ with the symplectic form $\omega_H= d(r-H)\wedge d\theta+\sigma$. What we will show is that under suitable replacements of $H$ by a Hamiltonian $K$ which have isomorphic Hamiltonian Floer homology on $B^{2n}$, that there is a corresponding isomorphism for the relevant symplectic homology groups as calculated on the local model  with the Hamiltonian $f$ and symplectic form $\omega_H$ to the associated Floer homology calculated on the local model  with the Hamiltonian $f$ and the symplectic form $\omega_K= d(r-K)\wedge d\theta +\sigma$.\\

The first of the two lemmas states that picking a suitable family $H^s$ which leaves $0\in B^{2n}$ uniformly isolated gives us a family of  symplectic forms $\omega_s=d(r-H^s)\wedge d\theta+\sigma , s\in [0,1]$ where the fixed point set $\gamma$ of $f$ with respect to $\omega_s$ remains uniformly isolated as we let $s$ vary from 0 to 1. This induces an isomorphism of the associated Floer homology groups to the original local symplectic homology groups even if the time-one map of $f$ with respect to $\omega_s$ changes throughout this path. The second Lemma shows that if we replace $H$ by some Hamiltonian $K$ such that $\phi^1_H=\phi^1_K$, then there is a symplectomorphism of the local model  $\psi$ fixing $r$ such that $\psi^*\omega_H=\omega_K$. This also induces an isomorphism of the associated Floer homology groups to the original symplectic homology groups. This second lemma can be thought of as showing that the symplectic homology of $\gamma$ is determined by the return map of the Reeb flow on a slice $\theta_0\times B^{2n}$.\\

\begin{Lemma}
\label{thm:lemma1}

Let $H^s_\theta$ be a family of one-periodic Hamiltonians where $s,\in [0,1], \theta \in S^1$ with the associated symplectic forms $\omega_s = d(r-H^s)\wedge d \theta+\sigma$ defined on $W =(1-\delta,1+\delta)\times S^1\times B^{2n}$ with coordinates $(r,\theta, z)$. Let $f(r)$ be Hamiltonian of the type used in Lemma \ref{thm:one_periodic_are_reeb}. If $0 \in B^{2n}$ remains a uniformly isolated fixed point of $\phi^1_{H^s}$ for all $s\in [0,1]$ and $\phi_{H^s}^\theta(0)=0$ for all $s,\theta\in [0,1]$ then the set $\gamma = 1\times S^1 \times 0$ is the fixed point set of the time-one flow of $f$ with respect to $\omega_s$and remains uniformly isolated for all $s$. 
\end{Lemma}

\begin{proof}

The time-t flow of $f(r)$ with respect to $\omega_s$ is 
\[
	\phi^t_{f,\omega_s} (r,\theta, z)= (r, f't+\theta, \phi_{H^s}^{f't+\theta}\circ(\phi_{H^s}^{\theta})^{-1}(z)).
\] In particular, we want to look at the time-one flow which is 
\[
	\phi^1_{\omega_s, f} (r,\theta, z)= (r, f'+\theta, \phi_{H^s}^{f'+\theta}\circ(\phi_{H^s}^{\theta})^{-1}(z)).
\] 
Here, we want to identify the fixed point set of the time-one map in order to show it remains uniformly isolated for all $s\in [0,1]$. First we look at the $\theta$ coordinate. In order to have a fixed point, we would have to have $f'+\theta=\theta$, which is true when $f'\in \mathbb Z$. Now $f'(r)=1$ only when $r=1$ by assumption. It is close to 1 otherwise. So in order to be a fixed point we must be in the set $1 \times S^1\times B^{2n}$ and on this set $f'=1$. So looking at the $z$ coordinate we find a fixed point if and only if $\phi_{H^s}^{1+\theta}\circ\phi_{H^s}^{-\theta}(z)=z$. An important fact here is that because each $H^s$ is one-periodic we have that $\phi^{1+\theta}_{H^s}=\phi^{\theta}_{H^s}\circ\phi^1_{H_s}$. So we can say on each $\theta$ slice that the time-one maps are conjugate to the time one map at $\theta=0$. That is $\phi_{H^s}^{1+\theta}\circ\phi_{H_s}^{-\theta}(z) = \phi_{H^s}^{\theta}\circ\phi_{H^s}^1\circ(\phi_{H^s}^{\theta})^{-1}(z).$ Now by assumption $\phi^1_{H^s}$ has $0 \in B^{2n}$ is a uniformly isolated fixed point for all $s\in [0,1]$. At each $\theta$ we are conjugating this map by a diffeomorphism so we have exactly one fixed point which is $\phi^\theta_{H^s}(0)$. This is taken to be 0 by assumption. Thus $ \phi_{H^s}^{\theta}\circ\phi_{H^s}^1\circ(\phi_{H^s}^{\theta})^{-1}(z) =z$ if and only if $z=0$. This shows that the fixed point set of the time-one flow of $f$ with respect to $\omega_s$ is in fact $\gamma$ for each $s$, so the fixed point set is uniformly isolated for all $s\in[0,1]$ as claimed.

\end{proof}

\begin{Corollary}
\label{l1c}
Under the conditions described in Lemma \ref{thm:lemma1} we have
\[
 HF^{loc}_*(\phi^1_{H^0})\cong HF^{loc}_*(\phi^1_{H^1}) \qquad \text{and} \qquad HF^{loc}_*(\gamma, \omega_{H^0},f) \cong HF^{loc}_*(\gamma, \omega_{H^1},f), 
\]
as well as the corresponding statements for $S^1$ equivariant Floer Homology.
\end{Corollary}
\begin{proof}
Both isomorphisms are a consequence of (LFH1), or in the equivariant setting (LEFH1). 
\end{proof}

As a short remark, this means that if we wish to modify our Hamiltonian $H $ on $B^{2n}$] via a path of Hamiltonians $H^s$ that leave the origin uniformly isolated, then we are automatically leaving $\gamma$ uniformly isolated with respect to $\omega_s$. We next prove Lemma \ref{thm:lemma2} which shows that the return map determines symplectic homology up to a shift in degree.\\

\begin{Lemma}
\label{thm:lemma2}

Let $H$ and $K$ be one-periodic Hamiltonians on $B^{2n}$ such that $\phi^1_H= \phi^1_K$, and let $\omega_H = d(r-H)\wedge d\theta+\sigma$ and $\omega_K = d(r-K)\wedge d\theta+\sigma$ be the associated symplectic forms defined on the model space $W=(1-\delta, 1+\delta)\times S^1\times B^{2n}$ with coordinates $(r,\theta,z)$. Then there exists a symplectomorphism $\psi:W \to W$ which fixes the $r$ coordinate such that $(\psi^{-1})^*\omega_K=\omega_H$ and  as such $(\psi^{-1})^*f(r)=f(r)$.
\end{Lemma}

\begin{proof}

First, a word on the strategy of this proof. We use the fact that $\phi_H^1$ and $\phi_K^1$ differ by a loop. We will take $\psi^\theta=\phi_F^\theta$ where $F$ is the time dependent Hamiltonian $H\# (K^{-1})_t$. That is to say $\psi^1= \phi_H^1\circ(\phi_K^1)^{-1}$. As we are interested in the time-one map we will call $\psi^1$ simply $\psi$ to clean up our notation. Now we would like to say that $(\psi^{-1})^*(\omega_K) = \omega_H$ where $\psi^{-1}$ is the inverse of $\psi$, but $\psi$ is a symplectomorphism of ($B^{2n},\sigma$), so we need to extend it to the local model,  which we do in the following way:
\[
	(r,\theta, z) \mapsto (r,\theta,\psi^\theta(z)).
\] 
With only minor abuse of notation we will refer to both $\psi:(1-\delta,1+\delta)\times S^1\times B^{2n}\to (1-\delta,1+\delta)\times S^1\times B^{2n}$ and $\psi^\theta:B^{2n}\to B^{2n}$ as $\psi$ interchangeably. Which map we are referring to should be clear from context. The symplectomorphism here preserves $r$ and $\theta$ coordinates and so it is clear that $(\psi^{-1})^*(f(r)) =f(r)$. It only remains to show that $(\psi^{-1})^*(\omega_K)=\omega_H$.\\*

Since $(\psi^{-1})^*$ is linear on differential forms we compute $(\psi^{-1})^*(\sigma)$ and $(\psi^{-1})^*(d(r-K)\wedge d\theta)$ separately.  We will first look at $(\psi^{-1})^*(\sigma)$. We know that this is entirely determined by restricting $\psi^{-1}$ to the $S^1\times B^{2n}$ coordinates as $\psi$ and $\psi^{-1}$ only depend on $\theta$ and $z$ (they leave $r$ fixed and none of the component functions depend on $r$). In this sense, we may look at $\psi :S^1\times B^{2n} \to S^1\times B^{2n}$ in order to show that $(\psi^{-1})^*(\sigma)= \sigma - dF\wedge d\theta$ (this reduces to an exercise in linear algebra). After which, we look at the expanded form of $dF$ by  using the composition formula for $F=H\# (K^{-1})_t$. From this we show $(\psi^{-1})^*(-dK \wedge d\theta) =(dF-dH)\wedge d\theta$. Finally, putting these two things together we can show $(\psi^{-1})^*\omega_K =\omega_H$. \\*

Now we show $(\psi^{-1})^*\sigma= \sigma - dF\wedge d\theta$. If we know a differential form along a submanifold as well as its behavior on the normal direction to that submanifold, the differential form is determined in a neighborhood of that submanifold. If we do this along a family of submanifolds that foliate $S^1\times B^{2n}$ we will determine the differential form on the whole space. The foliating family of submanifolds we looks at are the $\{\theta\}\times B^{2n}$ slices. We observe that if we take $v, w\in T_{(\theta, z)}B^{2n}$, where $T_{(\theta, z)}B^{2n}$ is the tangent space to the submanifold $\{\theta\}\times B^{2n}$ at the point $(\theta, z)$, then it is elementary to verify that $(\psi^{-1})^*\sigma_{(\theta, z)}(v, w)=((\psi^\theta)^{-1})^* \sigma_{z}(v, w) = \sigma_{z}(v, w)= \sigma_{(\theta, z)}(v, w)$. This is because $\sigma$ doesn't depend on the $\theta$ coordinate and for each $\theta$ we know that $\psi^{\theta}$ and its inverse are symplectomorphisms of $B^{2n}$ with the symplectic form $\sigma$. Next we describe what happens in the normal direction. Since $\sigma$ is a maximally nondegenerate two-form on an odd dimensional space we know it has a kernel. The kernel of $\sigma$ at any point is $\frac {\partial}{\partial \theta}$; we use this to find the kernel of $(\psi^{-1})^*\sigma$. We can compute that $\psi_*(\frac{d}{d\theta}) = \frac{d}{d\theta} +\frac{\partial}{\partial\theta}\psi^\theta= \frac{d}{d\theta}+ X_{F}$, from which it follows that $(\psi^{-1})_*(\frac{d}{d\theta}+ X_{F})=\frac{d}{d\theta}$. This is because the push forward of a diffeomorphism is the inverse of the push forward of the inverse of the diffeomorphism. Now from this we know that the span of $\frac{\partial}{\partial\theta} +X_F$ is the kernel of $(\psi^{-1})^*\sigma$. So with this we can deduce  

\begin{align}\label{l2eq1}
	(\psi^{-1})^*(\sigma) = \sigma - dF\wedge d\theta 
\end{align} and one can easily check that $\frac{\partial}{\partial \theta}+X_F$ is the kernel of this two form and that it has the same value on $v, w \in T_{(\theta, z)}B^{2n}$ as $\sigma$ does on those vectors. This is because $v \in T_{(\theta, z)}B^{2n}$ is in the kernel of $d\theta$. This proves the first half of our claim.\\*

Next we look at the formula for $H\# (K^{-1})$. Here we have that $(H\# (K^{-1}))_t=H_t+ (K^{-1})_t\circ(\phi_H^t)^{-1} =H_t- K_t\circ\phi_K^t\circ(\phi_H^t)^{-1}$. So at $t=1$ we have $F = H-K\circ(\psi^{-1})$. Thus $dF = dH-(\psi^{-1})^*(dK)$, and from this we can see $(\psi^{-1})^*(-dK) =dF-dH$. Hence it follows that 
 \begin{align}\label{l2eq2}
	(\psi^{-1})^*(-dK \wedge d\theta) =(dF-dH)\wedge d\theta
\end{align} which shows the second half of our claim.\\*

Finally we can compute $(\psi^{-1})^*\omega_K$. Here we have 
\[
	(\psi^{-1})^*{\omega_K} =(\psi^{-1})^*(dr\wedge d\theta -dK\wedge d\theta +\sigma) = dr\wedge d\theta +(\psi^{-1})^*(-dK\wedge d\theta) + (\psi^{-1})^*\sigma
\] 
Now applying equation \eqref{l2eq1} to  $(\psi^{-1})^*\sigma$, and equation \eqref{l2eq2}  to $(\psi^{-1})^*(-dK \wedge d\theta)$ we get
\[
	(\psi^{-1})^*{\omega_K} = dr\wedge d\theta +(dF-dH)\wedge d\theta +  \sigma - dF \wedge d\theta= d(r-H)\wedge d\theta +\sigma  = \omega_H.
\] 
This completes the proof.\\*

\end{proof}

\begin{Corollary}
\label{l2c}
Under the conditions described in Lemma \ref{thm:lemma2} we have
\[
 HF^{loc}_*(\phi^1_{H})\cong HF^{loc}_{*+s}(\phi^1_{K}) \qquad \text{and} \qquad HF^{loc}_*(\gamma, \omega_{H},f) \cong HF^{loc}_{*+s}(\gamma, \omega_{K},f),
\]
where $s$ is two times the Maslov index of the loop $\phi^t_F$ from Lemma \ref{thm:lemma2}, as well as the corresponding statements for $S^1$ equivariant Floer Homology.
\end{Corollary}

\begin{proof}
For the first isomorphism we can apply (LFH2) directly because $\phi^1_H=\phi^1_K$ implies that the Hamiltonian isotopies differ by a loop. The second isomorphism follow from the fact we have shown $W,\omega_H, f$ is symplectomorphic to $W, \omega_K, f$ and both Floer homology groups are obviously symplectomorphism invariants. Of course this is only up to a canonical isomorphism which may involve a shit in degree.\\

Now a few words about why the shift of degree is the same in both isomorphisms. For the first isomorphism we know that $\phi^t_H=\phi_F^t\circ\phi^t_K$, and from (LFH2) we have that $HF^{loc}_*(\phi_F^t\circ\phi^t_K)\cong HF^{loc}_{*-2\mu}(\phi^t_K)$ where $\mu$ is the Maslov index of the loop $\phi^t_F$. Now on the symplectic homology side we need to focus on the trivialization, which we will call $\xi$. We have shown that $(\psi^{-1})^*\omega_K=\omega_H$ and $(\psi^{-1})^*f=f$ but it is not necessarily the case that $(\psi^{-1})^*\xi =\xi$. Thus what immediately follows from the above Lemma is that
\[
	 HF^{loc}_{*}(\gamma, \omega_{K},f, \xi) \cong  HF^{loc}_*(\gamma, \omega_{H},f, (\psi^{-1})^*\xi) 
\] where the grading on the left is computed using $\xi$ and on the right using $(\psi^{-1})^*\xi$. So what remains is to compare $HF^{loc}_*(\gamma, \omega_{H},f, (\psi^{-1})^*\xi)$ and $HF^{loc}_*(\gamma, \omega_{H},f, \xi) $. This change in trivialization will produce a shift in degree equal to twice the Maslov index of the frame $(\psi^{-1})^*\xi$ relative to the frame $\xi$. We can pick the basis $\{\frac{\partial}{\partial r},\frac{\partial}{\partial \theta}, \frac{\partial}{\partial x_1},....,\frac{\partial}{\partial x_{2n}} \}$. In this basis we may express $(\psi^{-1})^*\xi$ as the matrix with columns $\{\frac{\partial}{\partial r},\frac{\partial}{\partial \theta} +X_F, d\phi^\theta_F \}$. From this together with homotopy invariance of the Maslov index we can see that the Maslov index here corresponds to that of the matrix $d\phi^\theta_F$, that is the Maslov index of $\phi_F^t$. Thus we have that  
\[
	HF^{loc}_*(\gamma, \omega_{H},f, (\psi^{-1})^*\xi) \cong HF^{loc}_{*+2\mu}(\gamma, \omega_{H},f,\xi). 
\] From this it follows that $ HF^{loc}_{*}(\gamma, \omega_{H},f, \xi) \cong  HF^{loc}_{*-2\mu}(\gamma, \omega_{K},f, \xi) $ which shows that we have the same shift in degree for both the local Hamiltonian Floer homology side of the isomorphism as well as the symplectic homology side.

\end{proof}

With these two tools in hand we now begin the proof in earnest.\\*

\section{A setting where the K{\"u}nneth formula may be applied}
\label{sec:prep for Kunneth}

In this section we go through the remaining steps required in order to arrive at a model that has homology isomorphic to the original while also satisfying the hypothesis needed to apply the K{\"u}nneth formula.The steps outlined here are necessary for the proof of both theorems and the steps apply to both the equivariant and non-equivariant settings. The first proposition in the section will allow us to replace our original $H$ by one which autonomous. The second proposition will then use a change of variable, as well as a final change to our Hamiltonian, finally putting us in the a situation where we will be able to apply the K{\"u}nneth formula for Hamiltonian Floer Homology.\\

\begin{Proposition}
\label{autonomous H}
Consider the local model  $W$ with the symplectic for $\omega_H$ and the Hamiltonian $f$ as described in Section \ref{new model}. Through a series of replacements of the types described in Lemma \ref{thm:lemma1} and Lemma \ref{thm:lemma2} we may replace our original Hamiltonian $H$ with a new Hamiltonian $\hat H$ such that
\[
HF_*^{loc}(\phi^1_H) \cong HF_{*+s}^{loc}(\phi^1_{\hat H}) \qquad \text{and} \qquad HF_*^{loc}(\gamma,\omega_H,f)\cong HF_{*+s}^{loc}(\gamma,\omega_{\hat H},f)
\] where $\omega_{\hat H} = d(r-\hat H)\wedge d\theta+\sigma$. Additionally the corresponding statement for equivariant Floer homology is true. 

Finally the new Hamiltonian $\hat H$ enjoys the following properties:

\begin{itemize}
\item{($\hat H$1)} $\hat H$ does not depend on $\theta$ it is an autonomous Hamiltonian. 
\item{($\hat H$2)} $\phi^t_{\hat H}$ has $0$ as an isolated fixed point of $B^{2n}$ for $t\in (0,2)$
\item{($\hat H$3)} $\hat H(0) =0$

\end{itemize}

\end{Proposition}

\begin{proof}
The proof of this proposition is somewhat long and technical. So as not to distract from the flow of the argument the proof is left for Section \ref{sec:propproof}.
\end{proof}

Now we make the final adjustments to our model.\\

As it currently stands the contribution of $\hat H$ to the flow comes from inside of the symplectic form $\omega_{\hat H}$. Our current Hamiltonian vector field at this point of the argument is that of $f(r)$ with respect to $\omega_{\hat H}$ which is $X_{\omega_{\hat H},f} = f'(r) (\frac{d}{d\theta} +X_{\hat H})$. Since our aim is to apply the K{\"u}nneth formula, we want to modify our system in such a way that the $X_{\hat H}$ term isn't being scaled by $f'(r)$. Additionally, we want to ensure we are not changing the relevant Floer homology groups up to isomorphism. In order to do this we will extract the Hamiltonian $\hat H$ from our symplectic form by using a change of coordinates. This is the content of our next proposition.\\

\begin{Proposition}
\label{prop:cov}
Consider the local model  $W$ with the symplectic form $\omega_{\hat H}$ from Proposition \ref{autonomous H} and the Hamiltonian $f(r)= \epsilon + (1-2\epsilon)r+\epsilon r^2$. After applying the change of variable $\rho = r-\hat H$ we have the following isomorphism in local Floer homology:
\[
	HF_*^{loc}(\gamma, \omega_{\hat H},f) \cong HF^{loc}_*(\gamma, d\rho\wedge d\theta + \sigma, f(\rho)+\hat H(z)),
\] where $f$ only depends on $\rho$ and $\hat H$ only depends on $z$. Additionally, the corresponding statement for equivariant Floer Homology is true. 
\end{Proposition}

Before we begin the proof we discuss our choice to take $f(r)= \epsilon + (1-2\epsilon)r+\epsilon r^2$. Up until now we had not made a specific choice for $f$, but as the symplectic homology is independent of $f$, providing that $f(1)=f'(1)=1$ and that $f''>0$, and as such we may choose whatever $f$ is most useful for our purposes. In the proof that follows we will want to expand $f$ in our change of coordinates and this quadratic $f$ will allow us to do this explicitly. Our choice for $f$ here satisfies all the necessary conditions for the definition of symplectic homology. If we additionally require that $2\epsilon\delta<\eta$ we can ensure the condition that $f'(c) \in (1-\eta,1+\eta)$ for $c\in (1-\delta,1+\delta)$, which was required for Lemma \ref{thm:one_periodic_are_reeb}, remains true.\\

\begin{proof}
We begin with the change of variable. First recall that $W= (1-\delta,1+\delta)\times S^1 \times B^{2n}$ with coordinates $(r,\theta, z)$.  As stated above we will take $\rho = r-\hat H$, giving us the change of variable $\Phi: W \to \Phi(W)$ where
\[
	(r,\theta, z)\mapsto (r-\hat H(z),\theta,z).
\] Under this change of coordinates, $d(r-\hat H)$ is $d\rho$, and $r=\rho+\hat H$, from which it follows $f(r) = f(\rho +\hat H)$ and our symplectic form becomes $d\rho\wedge d\theta +\sigma$. We see $\Phi$ is invertible and it clearly has full rank so it is a diffeomorphism onto its image.\\

Next we discuss what the image of the local model  $\Phi(W)$ is under this transformation. First the fixed point set $\gamma=1\times S^1\times0$ maps to itself under $\Phi$ because as mentioned in Proposition \ref{autonomous H} we have $\hat H(0)=0$, and thus on $\gamma$ we have $\rho=1$. As such, $\Phi(\gamma)=\gamma$, which we will need in later in the proof.  Furthermore, by potentially shrinking $\delta$ and the radius of $B^{2n}$ we can think of the original local model  as well as the space after the change of coordinates both containing a shrunken $W$ although the map $\Phi$ may not be bijective on this set. In particular if we let $a = \sup_{z\in B^{2n}}\hat H(z)$ then if $r=1-\delta$ we have $\rho \in ((1-\delta-a,1-\delta+a)$ and if $r=1+\delta$ we have $\rho \in ((1+\delta-a,1+\delta+a)$. This means if $a$ is small enough, which can be achieved by shrinking the radius of $B^{2n}$, we can say for instance $(1-\delta/2, 1+\delta/2)\times S^1\times B^{2n} \subset \Phi(W)$. Which is to say $\Phi$ surjects onto this set. Although this isn't terribly important to the proof that follows, the above set is an open neighborhood of our fixed point set. This open neighborhood lives in both $W$ and $\Phi(W)$, and as such we will choose to call this new set $W$. This allows us to unambiguously refer to $W$ independent of the chosen coordinates as it lies in both spaces. Additionally we will still say $W= (1-\delta,1+\delta)\times S^1\times B^{2n}$ even though we may have had to shrink this set.\\

Now we can rewrite $f(r)$ in our new coordinates
\begin{align*}
	f(\rho+\hat H) = \epsilon + (1-2\epsilon)(\rho +\hat H)+\epsilon(\rho +\hat H)^2\\
	= f(\rho)+(1-2\epsilon)\hat H+2\epsilon\rho \hat H+\epsilon \hat H^2
\end{align*} by substitution. Since we are pulling back our symplectic form and our function by $\Phi$ we of course have
\[
	HF_*^{loc}(\gamma, d(r-\hat H)\wedge d\theta +\sigma, f(r))  \cong HF_*^{loc}(\gamma, d\rho\wedge d\theta + \sigma,  f(\rho)+(1-2\epsilon)\hat H+2\epsilon\rho \hat H+\epsilon \hat H^2)
\]as well as the corresponding statement for equivariant homology. We are now set up to perform the final modification to our Hamiltonian.\\

To do this we look at the family of Hamiltonians
\[
	H^s =  f(\rho)+\hat H+\epsilon s ((2\rho-2) \hat H+\hat H^2). 
\] Here $H^0= f+\hat H$ and $H^1= f(\rho +\hat H)$. Our goal is to show our fixed point set $\gamma=1\times S^1\times 0$ remains uniformly isolated as we vary $s \in [0,1]$ in order to show our homology remains constant.\\

We actually have the ability to integrate our flow for each $s\in [0,1]$ which we may use to show that $\gamma$ remain the only fixed points of our Hamiltonian diffeomorphism for all $s\in [0,1]$. So we look at
\begin{align*}
	dH^s =  f'd\rho+d\hat H+2\epsilon  s(\hat H d\rho +(\rho-1) d\hat H+\hat Hd\hat H)\\
	= (f'+2\epsilon s\hat H)d\rho + (1+2\epsilon s((\rho-1) +\hat H))d\hat H.
\end{align*} Thus we see that the Hamiltonian vector field for $H^s$ with the symplectic form $d\rho\wedge d\theta +\sigma$ is given by 
\[
	X_{H^s}=  (f'+2\epsilon s\hat H)\frac{\partial}{\partial \theta} + (1+2\epsilon s((\rho-1) +\hat H))X_{\hat H}.
\] So the corresponding time-one flow of $\phi_{H^s}$ is the following map
\begin{align}
\label{cov flow H^s}
	(\rho, \theta, z) \mapsto (\rho,\theta + (f'+2\epsilon s\hat H),  \phi_{\hat H}^{(1+2\epsilon s((\rho-1) +\hat H))}(z))
\end{align} One might think that the last coordinate should be $\phi_{\hat H}^{(1+2\epsilon s((\rho-1) +\hat H)+\theta)} \circ (\phi^\theta_{\hat H})^{-1}(z)$ as it has appeared since the introduction of the new model. However, as stated in Proposition \ref{autonomous H}, $\hat H$ is autonomous  we have that 
\[
	 \phi_{\hat H}^{(1+2\epsilon s((\rho-1) +\hat H)+\theta)} \circ (\phi^\theta_{\hat H})^{-1}=  \phi_{\hat H}^{(1+2\epsilon s((\rho-1) +\hat H))} \circ \phi^\theta_{\hat H} \circ (\phi^\theta_{\hat H})^{-1}= \phi_{\hat H}^{(1+2\epsilon s((\rho-1) +\hat H))}. 
\] We need more than just this fact, however, before we show what the fixed points are.\\*

As such we  return our attention to a few properties of $\hat H$ and its associated Hamiltonian flow which were proven in Proposition \ref{autonomous H}. Namely $\hat H(0) = 0$,  and that $\phi^t_{\hat H}$ has $0$ as an isolated fixed point of $B^{2n}$ for $t\in (0,2)$. Now looking back at the equation in line (\ref{cov flow H^s}) with the above facts in hand we can find our fixed point set for any given $s\in [0,1]$. Since $\hat H$ is zero at the origin and $\rho$ is in $(1-\delta,1+\delta)$ we can ensure that $|2\epsilon ((\rho-1) +\hat H))|<1$ simply by shrinking $\delta$ and the radius of the $B^{2n}$ we are looking at. This has no effect whatsoever on our local homology. Doing so ensures that $(1+2\epsilon s((\rho-1) +\hat H)) \in (0, 2)$ for all $s\in [0,1]$. Which means that $z\in B^{2n}$ is only fixed by $\phi_{\hat H}^{ 1+2\epsilon s((\rho-1) +\hat H)}$ if $z$ is the origin. So if we want to look at the fixed points of $\phi_{H^s}$ it suffices to look at the set $(1-\delta,1+\delta)\times S^1\times 0$.\\

At this point we turn our attention to the other term of our flow $\theta +(f'+2\epsilon s\hat H)$ this is only a fixed point if $(f'+2\epsilon s\hat H)$ is an integer. However on the set $(1-\delta,1+\delta)\times S^1\times 0$ we are looking only at $f'$ because $\hat H(0)=0$. But $f'(1)=1$ and $f''(1)>0$, by definition, and thus $f'(\rho)=1$ only at $\rho = 1$ and is close to but not equal to 1 otherwise. Thus the only possible fixed points are those in the set $1 \times S^1 \times 0$. It is easy to verify that in fact every point in $\gamma$ is a fixed point by the logic laid out above. This is true for all $s\in [0,1]$ showing our fixed point set remains uniformly isolated and as such
\[
	 HF^{loc}(\gamma, d\rho\wedge d\theta+\sigma, f(\rho+\hat H)) \cong HF^{loc}(\gamma, d\rho\wedge d\theta+\sigma, f+\hat H)
\] as follows from (LFH1). The corresponding equivariant statement is also true and follows from the above argument together with (LEFH1).

\end{proof}

\section{Proof of the main theorems}
\label{sec:proofmain}

Before we move on to finally proving the our theorems we recall what has been shown up to this point. We have proven that 
\begin{align}
\label{LFHiso}
	SH_*(\gamma) \cong HF^{loc}_{*+s}(\gamma, d\rho\wedge d\theta +\sigma, f+ \hat H).
\end{align}We have also proved the corresponding statement for equivariant symplectic homology
\begin{align}
\label{EFHiso}
	SH_*^{S^1}(\gamma) \cong HF^{S^1}_{*+s}(\gamma, d\rho\wedge d\theta +\sigma, f+ \hat H). 
\end{align}The Hamiltonian $\hat H$ is an autonomous Hamiltonian such that 
\begin{align}
\label{B2niso}
	HF_*^{loc}(\phi^1_{H}) \cong HF^{loc}_{*+s}(\phi_{\hat H}^1)
\end{align} where $H$ is minus the original $H$ given from the model in \cite{HM:contacthom}. These facts together are everything we need to proceed.\\

\subsection{Proof of Theorem \ref{thm:iso1}}
\label{sec:proof1}
\begin{proof}
With our Hamiltonian finally in a suitable form we may at last apply the K{\"u}nneth formula for Hamiltonian Floer homology. After all the modifications, $f$ depends only on $\rho$  and $\hat H$ depends only on $z$, which enables us to apply K{\" u}nneth formula to obtain 
\[
	HF^{loc}_*(\gamma, d\rho \wedge d\theta+\sigma,\hat H+f) \cong \bigoplus_{i+j=*} HF^{loc}_i(0, \sigma,\hat H) \otimes HF^{loc}_j(S^1, d\rho\wedge d\theta,f).
\] 
We don't have any extra terms coming from the K{\"u}nneth spectral sequence which involve the higher Tor groups over the ring $R$ in which we compute the homology. This is because, as will become clear below, the local Floer homology groups of $HF^{loc}_*(S^1, d\rho\wedge d\theta,f)$ are free $R$-modules supported in exactly two degrees and thus all the higher Tor groups will vanish, making the spectral sequence collapse into the isomorphism above.\\

Looking at the left hand side of this tensor we note that $HF^{loc}_{*+s}(0, \sigma,\hat H) = HF^{loc}_{*+s}(\phi^1_{\hat H}) \cong HF^{loc}_{*}(\phi^1_H)$ where $H$ here is the negative of our original Hamiltonian. This was shown through several steps in the proof of Proposition \ref{autonomous H}, and it was done in such a way that we preserved the symplectic homology of $\gamma$ up to isomorphism, as stated above in equation (\ref{B2niso}). On the right hand side of the tensor we know that 
\[
	HF^{loc}_*(S^1, d\rho\wedge d\theta,f)  \cong  HM_{*+1}(S^1,f).
\] This follows from the fact that if $\alpha$ is a symplectic form, $F$ is a Morse-Bott nondegenerate fixed point set of $\phi_f$ a Hamiltonian diffeomorphism for an autonomous Hamiltonian $f$ with respect to $\alpha$, then $HF^{loc}_*(F,\alpha, f) = HM_{*+n}(F,f)$ where the right hand side is the Morse homology of $F$, and $n$ is half the dimension of the symplectic space in which homology is being computed. Additionally the grading on the right comes from the morse index of the isolated critical points one obtains when taking a subtable Morse perturbation of $f$. This is a fact that goes back to \cite{Po} and is well known in the literature.\\*

Combining the above statements with equation (\ref{LFHiso}) we get the following string of isomorphisms:
\begin{align*}
	SH_{*}(\gamma) \cong HF^{loc}_{*+s}(\gamma, d\rho \wedge d\theta+\sigma,\hat H+f)  \cong \bigoplus_{i+j=*+s} HF^{loc}_{i}(\phi^1_{\hat H}) \otimes HM_{j+1}(S^1,f) \\ \cong \bigoplus_{i+j=*} HF^{loc}_{i}(\phi^1_{H}) \otimes HM_{j+1}(S^1,f).
\end{align*}

 Finally we note that $HM_*(S^1,f)$ is only supported in degrees 0 and 1 as these are the Morse index of the critical points of a perturbation of $f$. So we can conclude
\[
	SH_*(\gamma) \cong HF^{loc}_*(\phi^1_H) \oplus  HF^{loc}_{*-1}(\phi^1_H),
\] which proves the claimed statement if we recall that $\phi_H$ is the return map $\phi$. \\*
\end{proof}

\subsection{Proof of Theorem \ref{thm:iso2}}
\label{sec:proof2}

\begin{proof}
Recall as stated in equation (\ref{EFHiso}) we have shown that $SH_*^{S^1}(\gamma) \cong \\ HF^{S^1}_{*+s}(\gamma, d\rho\wedge d\theta +\sigma, f+ \hat H)$. Now when we compute the equivariant symplectic homology we need to take the Hamiltonian $f(\rho)+\hat H$ defined on $(1-\delta,1+\delta)\times S^1\times B^{2n}$ and extend naturally to the space $S^1\times (1-\delta,1+\delta)\times S^1\times B^{2n} \times S^{2k+1}$ with coordinates $(\beta,\rho,\theta,z, \xi)$. That is $f+\hat H(\beta,\rho,\theta,z, \xi) = f(\rho)+\hat H(z)$. Then we need to make a suitable regular perturbation of this sum. In particular, we want to make a regular perturbation of our function of the form $ \tilde f(\beta,\rho,\theta,\xi)+ \tilde G(z)$, where $\tilde G(z)$ is a small nondegenerate perturbation of $\hat H$. Since $\hat H$ here may be degenerate when we take this perturbation our single fixed point may split into several. This is the definition of local Floer homology of $\hat H$ and we have $HF_*^{loc}(\phi^1_{\hat H}) \cong HF_*(\phi^1_{\tilde G})$. On the other hand $\tilde f$ depends on $\beta, \rho, \theta,$ and $\xi$ we know that $HF_*(\tilde f) \cong HF_*^{S^1}(f)$ by the definition of $S^1$-equivariant homology providing providing $2k+1$, the dimension of the sphere, is taken to be large enough that the $S^1$-equivariant Floer homology stabilizes in the specified degree.\\*

The point of making a perturbation of this type is that we can use the K{\"u}nneth formula once again. This is because $\tilde G$ only depends on the $z$ coordinates and $f$ depends on $\beta,\rho,\theta$, and $\xi$, but not on $z$. Applying K{\"u}nneth here we can see that 
\begin{align*}
	SH_{*}^{S^1}(\gamma) \cong HF^{S^1}_{*+s}(\gamma,d\rho\wedge d\theta+\sigma ,f(r)+\hat H(z) )\\
	=  HF_{*+s}(\tilde f(\beta,r,\theta,\xi)+ \tilde G(z) )\\
	\cong \bigoplus_{i+j=*+s} HF_i(\tilde f) \otimes HF_j(\phi^1_{\tilde G}).
\end{align*}
Again as is the case in the previous theorem the higher Tor groups which are used in computing the K{\"u}nneth spectral sequence will all vanish due to the fact that the homology groups of $HF_*(\tilde f)$ are free $R-$modules of finite rank concentrated in only one degree and the spectral sequence will collapse to the isomorphism above.\\

Now from equation (\ref{B2niso}) we know that $HF_{*-s}^{loc}(\phi^1_{ H})\cong HF_*^{loc}(\phi^1_{\hat H}) \cong HF_*(\phi^1_{\tilde G})$, and furthermore as stated above $HF_*(\tilde f) = HF_*^{S^1}(f)$. We can compute $HF_*^{S^1}(f)$ by  applying  Proposition 2.22 in \cite{GG:convex}. This proposition lays out conditions under which we may say that the equivariant Floer homology is isomorphic to the equivariant singular homology. We may apply this proposition because $1 \times S^1$ is a Morse-Bott nondegenerate fixed point set in the $\rho, \theta$ coordinates. From applying the proposition, we know that $HF^{S^1}_*(f)\cong H^{S^1}_*(S^1)$, and we know the $S^1$-equivariant homology of $S^1$ is supported only in degree zero as it can be computed using the Gysin sequence for singular homology. Thus we have $HF^{S^1}_0(f) = R$ where $R$ is the coefficient ring we are computing homology over and it is supported only in degree zero. From these two facts it follows that 
\[
	SH_*^{S^1}(\gamma) \cong \bigoplus_{i+j=*+s} HF^{S^1}_i(f) \otimes HF_j(\phi^1_{\tilde G}) \cong HF_*^{loc}(\phi^1_{H}).
\] Which, recalling that $\phi_H$ is the return map $\phi$, we see we have proved the desired isomorphism.\\

We have one more thing to resolve here, which is that we did not take a direct limit. However, since this is local homology, we know that in any given degree the homology will become stable for some large $2k+1$, the dimension of our sphere. This fact is discussed at various points throughout \cite{GG:convex} and follows from a related property for $S^1$-equivariant Hamiltonian Floer homology. So, we can prove this isomorphism in any degree via the above method and it is therefore true in all degrees. 
\end{proof}

\begin{Remark}
Much of this argument should go through for admissible iterated orbits as well.
\end{Remark}

\section{Proof of Proposition \ref{autonomous H}}
\label{sec:propproof}

In this section we prove that we can replace our original Hamiltonian $H$ with a new Hamiltonian $\hat H$ which has the properties 
\begin{itemize}
\item{($\hat H$1)} $\hat H$ does not depend on $\theta$; it is an autonomous Hamiltonian. 
\item{($\hat H$2)} $\phi^t_{\hat H}$ has $0$ as an isolated fixed point of $B^{2n}$ for $t\in (0,2)$
\item{($\hat H$3)} $\hat H(0) =0$
\end{itemize}

In order to achieve this goal we need to go through several steps. The first thing we must do is show that we can handle the degenerate and nondegenerate part of our Hamiltonian diffeomorphism $\phi^1_H$ separately, by splitting it into a totally degenerate Hamiltonian diffeomorphism generated by a Hamiltonian $\tilde H$ and a nondegenerate Hamiltonian diffeomorphism generated by a Hamiltonian $Q$. Recall that a Hamiltonian diffeomorphism is nondegenerate if its linearized return map does't have 1 as an eigenvalue, and that it is totally degenerate if all of its eigenvalues are 1. As it turns out it is fairly easy to show condition ($\hat H1$) can be attained in the nondegenerate case, and we can do this at the same time that we prove we can split our Hamiltonian diffeomorphism in Section \ref{sec:split diffeo}.\\

It will take a fair bit of work to show we have an autonomous replacement in the totally degenerate case. We will first show that we can replace $\tilde H$ by a Hamiltonian $K$, which is autonomous for short time and has the origin as an isolated fixed point for $t\in(0.1]$, as well as enjoying other useful technical properties. We  then leverage this $K$ to get a new Hamiltonian $\hat K$ which is actually autonomous. We will repeatedly encounter the issue that the known replacement techniques will give us a Hamiltonian which is not periodic. This is an issue as the Hamiltonian needs to be periodic for its corresponding symplectic form to be defined on $W$, so after the initial replacements we make necessary tweaks to get a periodic Hamiltonian. This is all done in Section \ref{degenerate H}. \\

In replacing our Hamiltonian $H$ we will need to make sure that we are not changing the local Hamiltonian Floer homology of $H$ up to isomorphism as well as not affecting the symplectic homology as derived from $\omega_H$. For the former we will use formal properties of Hamiltonian Floer homology. For the latter we will make repeated use of Lemma \ref{thm:lemma1} and Lemma \ref{thm:lemma2} in order to show that the symplectic form associated with the new Hamiltonian and the symplectic form associated with the old Hamiltonian compute the same homology with respect to $f$ on $W$. As such, we will additionally have to show our Hamiltonians satisfy the hypothesis of one of these lemmas, which may involve proving additional technical conditions at any given step.\\

As for the condition $(\hat H2)$, this will be proved naturally on the totally degenerate side as a consequence of how we go about making the totally degenerate part autonomous. However, we will still need to prove this on the nondegenerate side and we do so in Section \ref{nondegenerate H} after we have finished replacing the totally degenerate part. Finally $(\hat H3)$we essentially get for free and we resolve it at the very end.

\subsection{Splitting a Hamiltonian diffeomorphism}
\label{sec:split diffeo}

The goal of this section is to show a Hamiltonian diffeomorphism that fixes the origin can always be deformed into a split Hamiltonian diffeomorphism, one part of which is totally degenerate and one Hamiltonian diffeomorphism that is nondegenerate and, more importantly, it can be done in such a way that we get isomorphisms for all the relevant homology groups. We also show that the nondegenerate part can be generated by an autonomous quadratic Hamiltonian in the process.\\

\begin{Claim} 
\label{H splits}
We may replace our original Hamiltonian $H$ defined on $B^{2n}$ by a Hamiltonian $Q+\tilde H$ where $\phi^1_{Q+\tilde H}= (\phi^1_Q, \phi^1_{\tilde H})$, with $\phi^1_Q$ nondegenerate, $\phi^1_{\tilde H}$ totally degenerate and the following facts hold true:
\begin{itemize}
	\item{(Split 1)} $HF^{loc}_*(\phi^1_H) \cong HF^{loc}_{*+s}(\phi^1_{Q+\tilde H})$
	\item{(Split 2)} $HF^{loc}_*(\gamma, \omega_H , f) \cong HF_{*+s}^{loc}(\gamma,d(r-(Q+\tilde H))\wedge d \theta +\sigma, f)$
	\item{(Q1)} The Hamiltonian $Q$ is quadratic and thus autonomous. 
\end{itemize} 
\end{Claim}

We start by saying a few words about when we have a Hamiltonian diffeomorphism which splits as a totally degenerate part and a non-degenerate part and how this will allow us to make modifications to those respective parts individually without affecting the local Floer homology groups involved.\\*
 
 {\bf Split Diffeomorphism.} Assume that $\mathbb B^{2n}$ splits as the product of two symplectic vector spaces $V$ and $W$ where our Hamiltonian also splits as $H=H_V+H_W$. Here $H_V$ and  $H_W$ are functions only depending on the coordinates of $V$ and $W$ respectively and whose respective Hamiltonian flows fix the origin. Furthermore, assume that the time-one flow $\phi_{H_V}$ of $H_V$ is nondegenerate and that the time-one flow $\phi_{H_W}$ of $H_W$ is totally degenerate. Then by applying the K{\"u}nneth formula for Hamiltonian Floer homology we get  
 \[
 	 HF^{loc}_*(\phi^1_{H_V +H_W}) \cong \bigoplus_{n+k=*} HF^{loc}_n(\phi^1_{H_V}) \otimes HF^{loc}_k (\phi^1_{H_W})
 \] Now, if we find a way to replace $H_V$ with some new Hamiltonian $Q$ and $H_W$ with some $K$. We do this in such a way that $HF^{loc}_*(\phi^1_{H_V}) \cong HF^{loc}_*(\phi^1_{Q})$ and $HF^{loc}_*(\phi^1_{H_W}) \cong HF^{loc}_*(\phi^1_{K})$. This would allow us to say  
 \[
 	HF^{loc}_*(\phi^1_{H_V +H_W}) \cong HF^{loc}_*(\phi^1_{Q +K}) 
 \] by applying K{\"u}nneth again. Thus if we can show our Hamiltonian diffeomorphism is a split diffeomorphism we are able to replace a given Hamiltonian in one set of coordinates with a suitable new one without affecting our Hamiltonian Floer homology.\\*
 
 Now more generally if $H$ is not necessarily split, but the diffeomorphism $\phi_H$ splits as $(\phi_V,\phi_W)$ where $\phi_V$ is non-degenerate and $\phi_W$ is totally degenerate. Then $\phi_V$ and $\phi_W$ are germs of symplectomorphisms fixing the origin in Euclidean space. Because of this both $\phi_V$ and $\phi_W$ are Hamiltonian diffeomorphisms. Then as above denote by $H_V$ the Hamiltonian for $\phi_V$ and $H_W$ the Hamiltonian for $\phi_W$. Now it is not necessarily the case that $H=H_V+H_W$. However, the local Floer homology is defined by $\phi_H$ up to a shift in degree as described in (LFH2) and we have that 
 \[
 	HF^{loc}_*(\phi_H) = HF^{loc}_{*+l}(\phi_{H_V+H_W}).
 \] As such we can now treat $H$ as if it were split.\\*
 
 {\bf The General Case.} Let $\phi$ be the germ of a Hamiltonian diffeomorphism fixing $0\in B^{2n}$ and generated by $H$. For some decomposition $B^{2n}=V\times W$ the linearization $d\phi_0$ splits as the direct sum of a symplectic linear map on $V$ whose eigenvalues are all different from one and a symplectic linear map on $W$ with all eigenvalues equal to one. We will show that $\phi$ is homotopic to a split map via Hamiltonian diffeomorphisms with uniformly isolated fixed point at $0\in B^{2n}$ and linearization $d\phi_0$. We will denote said homotopy by $\phi_s$ where $s\in [0,1]$.\\*
	
	More specifically let $G_s$ be the Hamiltonian generating $\phi_s$ as its time-one map which is obtained, up to a  reparameterization, by concatenating the flow $\phi_{H}^t, t\in [0,1],$ with the homotopy $\phi_\xi, \xi\in[0,s]$. Then $0\in B^{2n}$ would be a uniformly isolated fixed point of $\phi_{G_s}$ for all $s \in [0.1]$. Thus by (LFH1) we would have
 \[
 	HF^{loc}_*(\phi_H) \cong HF^{loc}_{*}(\phi_{G_1}).
 \] Now here $\phi_{G_1}$ is split so we may perturb the time-1 map of $H$ and arrive at something which splits as desired.\\
 
 Now let us construct the symplectic homotopy $\phi_s$. Let $N_V$ and $N_W$ be Lagrangian compliments to the diagonals $\Delta_V\subset V\times\bar V$ and $\Delta_W\subset W\times \bar W$ respectively. We ask that $N_V$ is transverse to the graph of $d\phi_0|_V$ and $N_W$ is transverse to $d\phi_0|_W$. Then $N= N_V\times N_W$ is a Lagrangian compliment to the diagonal $\Delta \subset B^{2n}\times\bar B^{2n}$ transverse to $d\phi_0$ and hence to $\phi$ on a small neighborhood of $0 \in B^{2n}$. Denote by $F$ the generating function of $\phi$ with respect to $N$ on this neighborhood. Note that $0 \in B^{2n}$ is an isolated critical point of $F$ and that $d^2F_0$ is split.\\
 
 Now we will construct a family of functions $F_s, s\in[0,1],$ on a neighborhood of $0\in B^{2n}$ with $F_0=F$ and such that:
  \begin{itemize}
	\item$0 \in B^{2n}$ is a uniformly isolated critical point of $F_s, \forall s \in [0,1]$ 
	\item $d^2(F_s)_0=d^2F_0$ 
	 \item $F_1$ is split, That is to say that $F_1$ is the sum of some function $q$ on $V$ and another function $h$ on $W$ on a neighborhood of $0\in B^{2n}$.
 \end{itemize}

Once the family $F_s$ is constructed we define $\phi_s$ in the obvious way by identifying the graphs of $\phi_s$ and $dF_s$ in $B^{2n}\times\bar B^{2n}=N \times \Delta=T^*\Delta$. Note that the graph of $dF_s$ is transverse to the fibers of the projection $\pi: B^{2n}\times \bar B^{2n} \to B^{2n}$ near $0\in B^{2n}$ since $d^2(F_s)_0=d^2F_0$. Additionally the graph of $dF$, which is identified with the graph of $\phi$, is transverse to the fibers. Now note that in the decomposition we find below  we will have  $F_1=q+h,$ where the function $q$ is a non-degenerate quadratic form on $V$ (in fact $d^2q=d^2F_0|_V$) and $h$ is a function on $W$ with an isolated critical point at the origin.\\

So we need only to find $F_s$, and we argue as follows. First observe that, by the implicit function theorem, there exists near $0 \in B^{2n}$ a unique smooth map $\Phi:W \to V$ such that $\Phi(0)=0$ and $F|_{V \times w}$ has a critical point at $\Phi(w)$. Let $\Sigma$ be the graph of $\Phi$. Again by the implicit function theorem we can see that $d\Phi$ vanishes at the origin because $d^2F_0$ is split. As such $\Sigma$ is tangent to $W$ at $0 \in B^{2n}$. Now $F_s$ is constructed in two steps. First, we take an isotopy on a neighborhood of $0 \in B^{2n}$, fixing 0 and having the identity linearization at 0, to move $\Sigma$ to $W$.  This isotopy turns $F$ into a function, say $F_{0.5}$, such that $F_{0.5}|_{V \times w}$ has a non-degenerate critical point at $(0, w)$ for all $w$ near the origin. As the second step, we apply the parametric Morse lemma to $F_{0.5}|_{V \times w}$ to obtain a homotopy from $F_{0.5}$ to a function $F_1$ of the desired form $q + h$.\\*

{\bf Adapting to the Context of Symplectic Homology.} Now we explain how we apply this result in our specific case. First let's take $Q$ to be the Hamiltonian that corresponds to the generating function $q$ on the coordinates where our symplectomorphism are non-degenerate. Because the identification made between the neighborhood of the diagonal in $B^{2n}\times\bar B^{2n}$ and the cotangent bundle of the diagonal is a linear map, we can easily see a quadratic generating function corresponds to a quadratic Hamiltonian.  Thus $Q$ is quadratic and, in particular, autonomous. This is the function mentioned in (Q1). Next let's take $\tilde h$ to be the Hamiltonian corresponding to the totally degenerate coordinates where our generating function is $h$. Then we have just shown that  
\[
	HF^{loc}_*(\phi^1_{H}) \cong HF^{loc}_{*}(\phi^1_{G_1}). 
\] This is because as described above the family of generating function s all have the origin as an isolated fixed point and thus the associated Hamiltonian diffeomorphisms all have the origin as a uniformly isolated fixed point for all $s\in[0,1]$. Hence, we may apply (LFH1). \\

Next on the symplectic homology side we would like to apply Lemma \ref{thm:lemma1} to the family of Hamiltonian diffeomorphisms described above as $\phi_s$. Checking that we satisfy the hypothesis of Lemma \ref{thm:lemma1} we easily see that $\phi_s$ has the origin as a uniformly  isolated fixed point by construction. Secondly for a fixed $s_0\in [0,1]$ the Hamiltonian isotopy corresponding to $\phi_{s_0}$ is defined by concatenating the flow of $\phi^t_H$ for $t\in [01]$ with $\phi _{s}$ where $s\in [0,s_0]$. Now $\phi_H^t(0)=0$ by definition and every $\phi_s$ has zero as an isolated fixed point for $s\in [0,1]$ thus clearly $\phi_s(0)=0$. So we have that the isotopy for any $\phi_{s_0}$ fixes the origin for all time. The only remaining hypothesis is that the Hamiltonians $G_s$ corresponding to $\phi_s$ are 1-periodic for every $s$. This is not the case a priori so we have some work to do here.\\

Here we will introduce a trick that we will use several times thoughtout this final section to get a family $\hat G_s$ of one-periodic Hamiltonians such that $\phi^1_{\hat G_s} =\phi^1_{G_s}$ (in particular, for this case $\phi^1_{\hat G_1}$ splits)  and $\hat G_s$ satisfies all the assumptions of Lemma \ref{thm:lemma1}. Then we have $\phi^1_{\hat G_s} =\phi^1_{G_s}$, and we can apply Lemma \ref{thm:lemma2} at $s=0$ to get the following string of isomorphisms:
\[
	HF_*^{loc}(\gamma,\omega_H,f) \underset{Lemma \ref{thm:lemma2}} \cong HF_{*+l}^{loc}(\gamma,\omega_{\hat G_0},f)  \underset{Lemma \ref{thm:lemma1}}\cong HF_{*+l}^{loc}(\gamma,\omega_{\hat G_1},f),
\] where $\omega_{\hat G_s} =d(r-\hat G_s)\wedge d \theta + \sigma, s\in [0,1]$. \\

We first note that this shift in degree in homology comes from Lemma \ref{thm:lemma2}. Secondly, we have that $\phi_{G_1}^1$ splits and we will say that it splits with Hamiltonians $\tilde Q+\tilde H$, where $\tilde Q$ and  $\tilde H$ may differ from $Q$  and $\tilde h$ above because the Hamiltonian corresponding to $\hat G_1$ will be one-periodic whereas the originals from the $G_1$ may not be. The Hamiltonian $Q$ is important. It is quadratic and autonomous, both properties which $\tilde Q$ may not enjoy, and hence we will want to remember $Q$ as we want to bring it back into the picture soon. As for $\tilde h$ and $\tilde H$ at this stage we care much more that the totally degenerate part splits off as opposed to what it actually is.\\

Looking  at the Hamiltonian Floer homology on $B^{2n}$ side of things we have the following corresponding isomorphisms
\[
	HF_*^{loc}(\phi^1_H) \underset{(LFH2)} \cong HF_{*+l}^{loc}(\phi^1_{\hat G_0})  \underset{(LFH1)}\cong HF_{*+l}^{loc}(\phi^1_{\hat G_1}),
\]as long as we can find a suitable $\hat G_s$.\\

Now we find the desired $\hat G_s$. We let $\hat G_s(\theta,z) = \lambda'(\theta)  G_s(\lambda(\theta),z)$, where $\lambda$ is a smooth real valued function of one variable with the following properties: $\lambda'(0)=\lambda'(1)=0$ with $\lambda'>0$ otherwise, and $\int_0^1\lambda'(\theta) d\theta = 1$ namely $\lambda(1)=1$. This first condition on $\lambda'$ guarantees $\hat G_s$ is one-periodic as defined. The integration condition together with the definition for $\hat G_s$ give us that $\phi^1_{\hat G_s} =\phi^1_{G_s}$, which we can see by looking at 
\[
	\frac{\partial}{\partial \theta} \phi^{\lambda(\theta)}_{G_s}= \lambda'(\theta)X_{G_s(\lambda (\theta))}= X_{\hat G_s} =\frac{\partial}{\partial \theta} \phi^\theta_{\hat G_s}. 
\] So this is some map that follows the flow $\phi_{G_s}$ but reparameterizing time, the condition that $\lambda' >0$ except at two points ensures we are never standing still in time, yeilding the property that $0\in B^{2n}$ is a uniformly isolated fixed point set of $\phi^1_{\hat G_s}$ for all $s\in[0,1]$. These conditions also imply $\phi^\theta_{\hat G_s}(0)=0$.\\  

As a result, we know that $\hat G_s$ satisfies all the required assumptions of Lemma \ref{thm:lemma1} and we may apply it now. Next, a word on applying Lemma \ref{thm:lemma2}; we may apply this at $s=0$ since both $\hat G_0$ and $H$ are one-periodic and $\phi^1_H=\phi^1_{\hat G_0}$.\\

Let us finally return our attention to $\hat G_1 = \tilde Q+\tilde H$. Since $\phi_{\hat G_1}^1= \phi_{ G_1}^1$ and they are both split, we know that in particular $\phi^1_{Q} =\phi_{\tilde Q}^1$ and as such $\phi^1_{\hat G_1}= \phi^1_{Q+\tilde H}$. Additionally, we know that $Q+\tilde H$ and $\hat G_1$ are both one-periodic so may apply Lemma \ref{thm:lemma2} one last time to see that 
\[
	HF_{*}^{loc}(\gamma,\omega_{G_1},f)  \cong HF_{*+j}^{loc}(\gamma,d(r-(Q+\tilde H))\wedge d \theta +\sigma, f)
\] which proves (Split 2).We can then apply (LFH2) to get the corresponding isomorphism
\[
	HF_{*}^{loc}(\phi^1_{G_1})  \cong HF_{*+j}^{loc}(\phi^1_{Q+\tilde H})
\] which proves (Split1).The final item in Claim \ref{H splits} follows.\\

\subsection{The degenerate Hamiltonian}
\label{degenerate H}

Now that we have shown we can split our Hamiltonian diffeomorphism in such a way that allows us to work with the degenerate part and the nondegenerate part independently, we are ready to tackle the degenerate part. Our end goal here is to replace the Hamiltonian $\tilde H$ by something autonomous without changing any relevant homology up to isomorphism. Our first step is to once again look at generating functions in Section \ref{sec:GF} and this time we will find one for $\tilde H$. We use this generating function to find a Hamiltonian $\tilde F$ that has the same time-one map as $\tilde H$ but also has the origin as an isolated fixed point for its time-$t$ flow when $t \in (0,1]$. In the next step $\tilde F$ satisfies the conditions required to apply a theorem from [GCC] which produces a periodic Hamiltonian $K$  which is autonomous for short time. We will show that we can replace $\tilde H$ by $K$ and not affect any relevant homology. This is done in Section \ref{find K}. \\

Finally in Section \ref{autonomous K} we use the Hamiltonian $K$ which posses certain properties allowing us to deform the Hamiltonian $Q+K$ into a new Hamiltonian which is autonomous proving $(\hat H1)$. The new autonomous Hamiltonian also lets us show that $(\hat H2)$ can be made true on the totally degenerate side.\\

\begin{Claim}
\label{H to K}
We may replace $\tilde H(\theta, z)$ by some Hamiltonian $K$ that has a number of technical properties which we will apply to obtain isomorphisms on homology. These properties are:
\begin{itemize}
	\item{(K1)} $\phi^1_{\tilde H}=\phi^1_K$
 	\item{(K2)}  $K$ is one-periodic.
 	\item{(K3)} $\phi^t_K$ is autonomous for small $t$, say $t\in [0,\epsilon)$. 
	\item{(K4)}  For $\phi^t_K$ we have $0 \in. B^{2n}$ is an isolated fixed point for all $t \in (0,1]$
\end{itemize}
\end{Claim}

In order to do this is we first find the Hamiltonian $\tilde F$ mentioned above. The proof that $\tilde F$ exists uses the technique of generating functions. After producing $\tilde F$  we will end up with something that is not periodic, and therefore is not a candidate to replace $\tilde H$. We will therefore need to make another replacement which is where we introduce the aforementioned Hamiltonian $K$ satisfying the desired technical conditions.\\

\subsubsection{The Generating Function}
\label{sec:GF}
  
In order to find $\tilde F$ in a way that ensures that $0 \in B^{2n}$ is the only fixed point of $\phi_{\tilde F}^t$  for all $t \in (0,1]$ and $\phi_{\tilde H}^1 =\phi^1_{\tilde F}$, we use generating functions completely analogous to the approach taken in section 6 of \cite{Gi:CC}. Since $0 \in B^{2n}$ is an isolated fixed point of $\phi^1_{\tilde H}$, and we are interested in the local Floer homology, we only need a generating function which is defined in a small neighborhood of the origin. This is a standard technique for producing a generating function which we have used once already in this paper. First, let us look at the twisted product $B^{2m} \times \bar B^{2m}$ with the symplectic form $\sigma \oplus -\sigma$. Here the diagonal $\Delta$ is identified with $B^{2m}$ and we call $\Gamma = (z,\phi_{\tilde H}^1(z)) \subset B^{2m} \times \bar B^{2m}$ the graph of $\phi_{\tilde H}^1$. In this symplectic space both $\Gamma$ and $\Delta$ are Lagrangian submanifolds. We want to pick some Lagrangian complement  $N$ such that $N$ is a subspace transverse to both $d\phi_{\tilde H}^1$ at the origin and $\Delta$ which is simply an exercise in symplectic linear algebra. With $N$ being transverse to $d(\phi_{\tilde H}^1)_0$ we ensure that $N$ is transverse to $\Gamma$ in some neighborhood of the origin and we have satisfied the needed conditions to have a generating function defined. With this normal Lagrangian fixed, we may now identify $N\times \Delta$ with $T^*B^{2m}$. Under this identification, we have a generating function $F$ corresponding to our Hamiltonian in the following sense: $dF=\Gamma$. Then we take our replacement for $\tilde F$ to be the Hamiltonian whose time-$t$ flows have graphs corresponding to $tdF$. The resulting generating function has the following properties:\\

\begin{itemize}
	\item{(GF1)} A point $p$ is an isolated critical point of $F$ if and only if it's an isolated fixed point of $\phi^1_{\tilde H}$
	\item{(GF2)} $||d^2F_0|| =||d(\phi^1_{\tilde H})_0-id||$ 
 \end{itemize}
 
\subsubsection{Finding the Hamiltonian K}
\label{find K}

Now we wish to replace $\tilde H$ with some Hamiltonian $K(\theta,z)$, which is autonomous for small $\theta$, has $0 \in B^{2n}$ as an isolated fixed point for $t\in (0,1]$, and where $K$ is periodic ensuring that $K$ lives on our local model . In a setting where we have a generating function $F$ which satisfies the conditions (GF1) and (GF2) we may apply lemma 6.1 in \cite{Gi:CC}. This gives us a Hamiltonian $K$ that has the following properties: 
\begin{itemize}
	\item{(K.1)} $\phi^1_K= \phi^1_{\tilde F}$ (which is in turn equal to $\phi^1_{\tilde H}$)
	\item{(K.2)} $K$ is one-periodic in time
	\item{(K.3)} $K$ is autonomous for short time
	\item{(K.4)} $0 \in B^{2n}$ is an isolated critical point of $K$ and $K(0)=0$ for all $\theta \in S^1$
	\item{(K.5)} $||d^2K_0||= O(||d^2F_0||)$
	\item{(K.6)}$ ||X_F-X_K||  \le (O(||d^2F_0||)+ O_\phi(r)) \cdot ||X_F||$ 
	\item{(K.7)} $|| \frac{\partial}{\partial \theta} X_K||\le (O(||d^2F_0||)+ O_\phi(r)) \cdot ||X_F||.$
\end{itemize}

Here by $ ||X_F-X_K||  \le (O(||d^2F_0||)+ O_\phi(r)) \cdot ||X_F||$, we mean there exists some $C_1$ independent of $\phi^1_{\tilde H}$ and some $C_2(\phi^1_{\tilde H})$ depending on $\phi^1_{\tilde H}$ such that 
\[
	 ||X_F-X_K||  \le (C_1||d^2F_0||+ C_2(\phi^1_{\tilde H})r) \cdot ||X_F||.
\]

The conditions (K.1), (K.2),(K.3) correspond exactly to (K1), (K2), and (K3). The condition (K.4) implies the Hamiltonian vector field $X_K$ has an isolated critical point at the origin for all time. The task is then to leverage this and the remaining conditions to prove (K4)

The technical conditions (K.4), (K.6) and (K.7) will allow us to adapt a proof by Hofer and Zehnder to show that the only fixed points of $\phi^\theta_K$ for short time are zeroes of the Hamiltonian vector field $X_K$. We then show we can make this change from $\tilde F$ to $K$ in such a way that $0\in B^{2n}$ remains an isolated fixed point for all time in $(0,1]$. \\

 This is a lemma about time dependent flows that gives a condition that ensures we have no fixed points other than the origin for short time. This proof is adapted from one presented in \cite{HZ}. 
 
 \begin{Lemma}
 Let $Y_t$ be a time dependent vector field on $B^{2n}$ with:
 \begin{itemize}
 	\item (i) $Y_t(0)$ is an isolated 0 of $Y_t \forall t $
	\item (ii)$ ||\frac{\partial}{\partial t} Y_t|| \le \epsilon||Y_0||$
\end{itemize}
Then the flow of $Y_t$ has no fixed points other than the origin for short time. The specific bound on time is given in the proof.
\end{Lemma}

Condition (ii) here is analogous to conditions (K.5) above and has been called being nearly-autonomous by some authors including N. Hingston when she introduced this idea in \cite{Hi}. 

\begin{proof}
Our aim now is to show there is some $t_0>0$ such that for $T<t_o$ we have that the time $T$ flow of $Y_t$ has no nontrivial fixed points, and hence the only fixed point is the origin by (i). Assume that $y(t)$ is a closed characteristic of $Y_t$ of period $T$. We may normalize to having period-one by taking $x(t)=y(Tt)$. Then we have 

\begin{align*}
	x(0)=x(1) && \dot x=TY_t(x).
\end{align*}

Next we note that by applying the vector valued mean value theorem we have that $||Y_t-Y_0|| \le \epsilon ||Y_0|| |t-0|$. And thus when time is bounded, as it is in our setting,  we have 
\begin{align}
\label{(iii)}
	 ||Y_t-Y_0|| \le \epsilon ||Y_0||.
\end{align} 
This condition corresponds to (K.6) above and we will use it shortly.\\

Recall for any time-one curve $c(t)$ with zero mean that by examining the Fourier expansion of $c(t)$ we have the $L^2$ estimate $||c|| \le\frac{1}{2\pi}||\dot c||$. Now as defined $\dot x$ has zero mean so we may look at 
\begin{align}
\label{path estimate}
	||\dot x||\le \frac{1}{2\pi}||\ddot x|| =  \frac{T}{2\pi}||\frac{\partial}{\partial t}Y_t + DY_t\cdot \dot x|| \le  \frac{T}{2\pi}\left (||\frac{\partial}{\partial t}Y_t||+ ||DY_t||\cdot||\dot x||\right).
\end{align}
Now by (ii) we have $ ||\frac{\partial}{\partial t} Y_t|| \le \epsilon||Y_0||$.  Then by equation (\ref{(iii)}) we have 
\[
	\left| ||Y_t||-||Y_0||  \right| \le ||Y_0-Y_t||\le \epsilon ||Y_0||
\] and hence
\[
	- \epsilon ||Y_0|| \le ||Y_t||-||Y_0||\le  \epsilon ||Y_0||.
\]From this it follows that 
\begin{align*}
	||Y_0||- \epsilon ||Y_0|| \le ||Y_t|| &\qquad\text{thus} & ||Y_0|| \le \frac{1}{1-\epsilon} ||Y_t||.
\end{align*} Now putting this all together we have 
\begin{align}
\label{VF bound}
	||\frac{\partial}{\partial t} Y_t|| \le \epsilon||Y_0|| \le  \frac{\epsilon}{1-\epsilon} ||Y_t|| =  \frac{\epsilon}{1-\epsilon} ||\dot x|| .
\end{align}
Now, substituting line (\ref{VF bound}) into the equation in line (\ref{path estimate}) we can see that 
\[
	||\dot x|| \le \frac{T}{2\pi}\left (\frac{\epsilon}{1-\epsilon} ||\dot x|| + ||DY_t||\cdot||\dot x||\right).
\] So if we have some $a$ as some upper bound for $||DY_t||$ then it follows that 
\[
	\frac{2\pi}{\frac{\epsilon}{1-\epsilon} +a}||\dot x|| \le T||\dot x||
\] From this we may conclude that if $T< \frac{2\pi}{\frac{\epsilon}{1-\epsilon} +a}$ that $\dot x =0$ implying that $x$ is a constant characteristic of $Y_t$. 
\end{proof}

Now there is one additional thing that we have to look at here, which is that in our case we want  at the flow of $X_K$ to have only trivial fixed points for $t\in(0,1]$. We can see that flow of $X_K$ has no nontrivial orbits of period  $T< \frac{2\pi}{\frac{\epsilon}{1-\epsilon} +a}$. But we want a Hamiltonian whose associated flow has the origin as an isolated fixed point for all time $t \in (0,1]$. So we need to show that  $\frac{\epsilon}{1-\epsilon} +a$ can be made less than $2\pi$ for our new Hamiltonian and then the claim (K4) will be proved.\\

Here we will show that we can make $\frac{\epsilon}{1-\epsilon}$ arbitrarily small if we can make we can make $||d(\phi^1_{\tilde H})_0-id||$ small. To this end it is sufficient if we show $\epsilon$ can be made small if $||d(\phi^1_{\tilde H})_0-id||$ is small. The statements in (K.6) and (K.7) together determine $\epsilon$ in terms of $(O(||d^2F_0||)+ O_\phi(r))$. That is $\epsilon$ can be made small if the quantity $(O(||d^2F_0||)+ O_\phi(r))$ can be made small. By looking at this in terms of (GF2) we see this quantity can be written as  $(O(||d(\phi^1_{\tilde H})_0-id||)+ O_\phi(r))$. We can make this small by shrinking the size of our ball, and are able to ensure that $||d(\phi^1_{\tilde H})_0-id||$ can be made as small as we like. We will address this shortly, but first we look at what is required in order for $a$ in the upper bound mentioned above to be made small.\\

In the Hamiltonian setting the object corresponding to $||DY_t(x)||$ is the quantity $||\nabla^2K(x)||$ which itself is bounded above by $\max ||d^2K_0|| +O(r)$. Then (K.5) and (GF2) together give us that $||\nabla^2K(x)|| = O(||d(\phi^1_{\tilde H})_0-id||)+ O(r)$. This is all to say that our upper bound $a$ can be made small if we shrink $r$ and can show that $||d(\phi^1_{\tilde H})_0-id||$ can be made small.\\

So for our final step we give some justification as to why $||d(\phi^1_{\tilde H})_0-id||$ can be taken to be small. This uses ideas presented in \cite{Gi:CC} Lemma 5.1, which states that by taking a change of coordinates we can assure that a totally degenerate symplectomorphism (or more specifically a conjugation of said symplectomorphism) can be produced so that its linearization at 0 is arbitrarily close to the identity. This is the condition we need to satisfy. This completes the proof of $(K4)$, and thus we have our candidate $K$ which will replace our original Hamiltonian $\tilde H$. \\

Finally, we say a few words on replacing $\tilde H$ with $K$.  We know $HF_*^{loc}(\phi^1_{\tilde H}) \cong HF_{*+s}^{loc}(\phi^1_{K})$ because they have the same time one map as stated in (K1) and we may apply (LFH2). Now what we are interested in $HF_*^{loc}(\phi^1_{Q+\tilde H})$ and $HF_*^{loc}(\phi^1_{Q+K})$. Both of these maps are split diffeomorphisms and in particular we have $\phi^1_{Q+\tilde H}= (\phi^1_{Q},\phi^1_{\tilde H}) = (\phi^1_{Q},\phi^1_{K})= \phi^1_{Q+K}$ so they have the same time one map additionally $Q+K$ is one-periodic because $K$ is as stated in (K2). Thus we may apply Lemma \ref{thm:lemma2} to obtain $HF_*^{loc}(\phi^1_{Q+\tilde H}) \cong HF_{*+s}^{loc}(\phi^1_{Q+K})$, as well as the corresponding statement for equivariant homology. Thus all desired isomorphisms have been proven.\\

\subsubsection{Producing an autonomous Hamiltonian}  
\label{autonomous K}

Now that we have replaced our degenerate Hamiltonian $\tilde H$ by some $K$ that has the properties outlined in Claim \ref{H to K}, namely that it is autonomous for short time, it is periodic, and has $0 \in B^{2n}$ is an isolated fixed point for all $t\in (0.1]$, we are interested in $\omega_{Q+K} = d(r-(Q+K))\wedge d\theta+\sigma$ which is the new symplectic form on $W$. Now the Hamiltonian flow of $f(r)$ with respect to $\omega_{Q+K}$ is
\begin{align*}
	\phi^t_f (r,\theta,z)= (r, f't+\theta, \phi_{Q+K}^{f't+\theta}(\phi_{Q+K}^\theta)^{-1}(z)) \qquad \qquad \qquad \,\,\,\\
	= (r, f't+\theta, \phi_{Q}^{f't+\theta}(\phi_{Q}^\theta)^{-1}(v),\phi_{K}^{f't+\theta}(\phi_{K}^\theta)^{-1}(w)),
\end{align*}
where $v$ are the coordinates on which we are nondegenerate and $w$ are the coordinates in which we are totally degenerate as described in Claim \ref{H splits}.\\

We want to replace the time dependent flow $\phi_K^t$ used here with something that is autonomous in such a way that our homology groups are all preserved up to isomorphism. The tools we will use are (LFH1) and Lemma \ref{thm:lemma1}. In order to do this we will give an $s$ dependent family $K^s$ where $s\in [t_0,1]$ where $t_0$ will be explained below, such that the hypothesis to apply the above facts are satisfied.

\begin{Claim}
\label{a nice K family}

Let $[0,2t_0]$ be a time interval on which $K$ is autonomous; this is guaranteed by (K3). There exists a family $K^s$ of Hamiltonians, with $s\in [t_0,1]$,that has the following properties:
\begin{itemize}
	\item (KS1) $K^s$ is periodic for all $s\in [t_0,1]$
	\item (KS2) $\phi^1_{K^s}$ has the fixed point set $0\in B^{2n}$ remaining uniformly isolated for all $s\in [t_0,1]$
	\item (KS3) $\phi^\theta_{K^s}(0)=0$ for all $s\in [t_0,1]$,$\theta\in [0,1]$
	\item (KS4) $K^{t_0}$ is an autonomous Hamiltonian. 
\end{itemize}
 
\end{Claim}

First let $\bar K^s = sK(s\theta, z)$. Then one can easily check that $\phi_K^s = \phi^1_{\bar K^s}$ because 
\[
	\frac{\partial}{\partial \theta} \phi^{s\theta}_K = sX_{K_{s\theta}} = X_{\bar K^s}.
\]
	This allows us to replaces the time-one Hamiltonian flow of $K$ by the time-s Hamiltonian flow of $K$. Since $K$ is autonomous for small time we can use this trick to replace our $\theta$ dependent Hamiltonian by one which is autonomous. Again, the interval on which $K$ is autonomous is $[0,2t_0]$. So we take $s\in [t_0,1]$ to get our path of Hamiltonians $\bar K^s$. This satisfies some of the desired conditions, but our new Hamiltonians are not one-periodic when $s\neq1$. This is a problem as it will not live on our local model  and we can't apply Lemma \ref{thm:lemma1}. So we address this next. We let $K^s(\theta,z) = \lambda_s'(\theta) \bar K^s(\lambda_s(\theta),z)$. Here we have a potentially different function $\lambda_s: [0,1] \to [0,1]$ for each $s$. We ask that the collection has the following properties 

\begin{itemize}	
	\item For  all $s\in [t_0,1]$, we have $\lambda_s$ is smooth in $s$ and $\theta$. 
	\item For all $s\in [t_0,1]$ we ask $\lambda_s'(0)=\lambda_s'(1)$ with $\lambda'_s>0$ otherwise, and $\int_0^1\lambda_s'(\theta) d\theta = 	1$ implying $\lambda_s(1)=1$. 
	\item Then we add conditions which depend on $s$. For $s\in[2t_0, 1]$ we ask that $\lambda_s'(0)=\lambda_s'(1)=0$ and for $s\in (t_0,2t_0)$ we relax this to earlier 		condition $\lambda_s'(0)= \lambda_s'(1)$ and we ask that $\lambda_{t_0}'=1$ the constant function.
\end{itemize}

 For any $s\in [t_0,1]$ it is easy to see that have not only that $\phi^1_{K^s}=  \phi^1_{\bar K^s}=\phi_K^s$, but we also have that 
\[
	\frac{\partial}{\partial \theta} \phi^{\lambda_s(\theta)}_{\bar K^s }= \lambda_s'(\theta)X_{\bar K^s(\lambda_s (\theta))}= X_{K^s} =	\frac{\partial}{\partial \theta} \phi^\theta_{K^s }.
\] 

Now we say a few words about proving the statements in Claim \ref{a nice K family}. Let's first take a look at (KS1); this condition is handled by the fact that $\lambda_s'(0)=\lambda_s'(1)$. First that $K^s$ is periodic from $[2t_0,1]$ and that is periodic from $[t_0,2t_0)$. This condition is handled in both cases in some sense by the condition that $\lambda_s'(0)=\lambda_s'(1)$. On  $[2t_0,1]$  we require that $\lambda_s'(0)=\lambda_s'(1)=0$ which makes our Hamiltonian periodic by making the Hamiltonian zero at time zero and at time one. On the interval $[t_0,2t_0]$ the Hamiltonian $K$ is autonomous so the only change that can occur to $K^s$ is the multiplication by the factor $\lambda_s'$ which we get from rescalling time. By requiring that  $\lambda_s'(0)=\lambda_s'(1)$, however, we get that it has the same scaling factor at time zero and time one and $K$ is autonomous so we have achieved something periodic.\\

Next we address (KS2) and (KS3). For $\bar K^s$ we know the fixed point set remains uniformly isolated because we are only flowing with $K$ up to time some time $s$ and by construction $0 \in B^{2n}$ is an isolated fixed point of $\phi^t_K$ for $t \in (0, 1]$. The flow of $K^s$ is following the flow of $\bar K^s$ except with a rescalling of time. The condition that $\lambda_s'>0$ except possibly when $t=0$ or $t=1$ ensures that this rescalling never has us standing still and a such we introduce no new fixed points which means that $0\in B^{2n}$ remains a uniformly isolated fixed point of $\phi^1_{K^s}$. Further more by the essentially the same reasoning we can see that $\phi^\theta_{K^s}(0)=0$ as we are following the flow of $K$ which maps $0 \in B^{2n}$ to itself for all time.\\*

Finally we address (KS4), which is a little more obvious than the proceeding points. Recalling that $K$ is autonomous from $t \in [0,2t_0]$ and that $K^s$ follows the flow of $K$ up to time $s$, it is simply $K$ multiplied by the scaling factor $\lambda_s'$. But we require that $\lambda_{t_0}'=1$. From this it follows that we are just flowing with $K$ from time zero to time $t_0$, and as a result we have an autonomous Hamiltonian. Lets call this autonomous replacement Hamiltonian $\hat K$ as we continue. This completes the proof of Claim \ref{a nice K family}.\\

Now drawing our attention to $Q+K^s$, condition (KS1) shows that $Q+K^^s$ is periodic, (KS2) gives us that $\phi^1_{Q+K^s}$ has $0 \in B^{2n}$ as a uniformly isolated fixed point for all $s \in [t_0,1]$, and (KS3) give us that $\phi^\theta_{Q+K^s}(0)=0$  for all $s,\theta \in [t_0,1]$. As such we may now apply (LFH1) and Lemma \ref{thm:lemma1} to the family $Q+K^s$ and we have the two following isomorphisms:

\begin{align*}
	HF_*^{loc}(\gamma,d(r-(Q+K))\wedge d\theta +\sigma ,f) \cong HF_*^{loc}(\gamma,d(r-(Q+\hat K))\wedge d\theta +\sigma ,f)\\
	 HF_*^{loc}(\phi^1_{Q+K}) \cong HF_*^{loc}(\phi^1_{Q+\hat K}).\qquad\qquad\qquad\qquad\qquad
\end{align*}

At this point we have finally shown ($\hat H$1) to be true, but we still need to prove condition ($\hat H$2). By construction, the time-one flow of $X_{\hat K}$ is the time $t_0$ flow of $X_K$ and the time-two flow of $X_{\hat K}$ is the time $2t_0$ flow of $X_K$, which shows ($\hat H$2) for the totally degenerate part. We next show that we can make a suitable replacement on the nondegenerate side.

\subsection{The nondegenerate Hamiltonian}
\label{nondegenerate H}

Here we modify the nondegenerate Hamiltonian $Q$ in order to finally prove condition ($\hat H2$). Since $Q$ is quadratic the time-one flow is just a symplectic linear transformation of $B^{2m}$ that has the origin as its only fixed point. We want to replace $Q$ by some autonomous $\hat Q$ such that $\phi^1_{Q}=\phi^1_{\hat Q}$, where $\phi^t_{\hat Q}$ has $0\in B^{2m}$ as an isolated fixed point for all $t\in (0,2)$.\\ 

Finding $\hat Q$ requires that we decompose the map $\phi^1_{Q}$ in terms of the eigenvalues of its linearized return map. It is helpful here to look at what $\phi^t_Q$ actually is. Our function $Q$ is a quadratic autonomous Hamiltonian which means we are solving a fairly simple first order ordinary differential equation when we compute the flow. The flow $\phi^t_Q(z)$ is equal to $\exp(Jd^2Qt)z$, where $d^2Q$ is the Hessian of $Q$ with respect to the $v$ coordinates, $J$ is the matrix representing multiplication by $i$ on $B^{2m}$, and the map $\exp: \mathfrak{sp}(2m) \to SP(2m)$ is the exponential map into the group of symplectic matrices. The matrix $Jd^2Q$ can be put in a normal form from a symplectic linear change of coordinates. This is a totally standard normal form, the details of which can be seen in \cite{Ar}. After this change of coordinates the matrix $Jd^2Q$ can be written in a block form. As such, the exponential of $Jd^2Q$ also has a block form and we can write $\phi^1_Q$ as the product of a Hamiltonian loop $\phi$ and the direct sum $\Phi_e \oplus \Phi_h$, where $\Phi_h$ is a Hamiltonian diffeomorphism for a Hamiltonian $Q_h$which is hyperbolic with real or complex eigenvalues. Here, $\Phi_e$ is a Hamiltonian diffeomorphism for some Hamiltonian $Q_e$ which is elliptic and up to conjugation it decomposes as a direct sum of short rotations $R_\theta$ where $\theta \in (-\pi,\pi)$. In particular, we can see that $\Phi_e\oplus\Phi_h$ is the time-one flow of $(\exp(Jd^2Q_et) , \exp(Jd^2Q_ht))$ for which $0\in B^{2m}$ remains uniformly isolated for all $t\in (0,1]$. We take $\hat Q$ to be the Hamiltonian $Q_e\oplus Q_h$. Furthermore by virtue of how $\hat Q$ is constructed, the elliptic part is less than a half rotation, and so we actually have that the origin is an isolated fixed point for all $t\in (0,2)$. Now finally we have $\phi^1_Q=\phi^1_{\hat Q}$ where both $Q$ and $\hat Q$ are autonomous , and thus one -periodic so we may apply (LFH2) and Lemma \ref{thm:lemma2} to see that

\begin{align*}
	HF_*^{loc}(\gamma,d(r-(Q+\hat K))\wedge d\theta +\sigma ,f) \cong HF_{*+s}^{loc}(\gamma,d(r-(\hat Q+\hat K))\wedge d\theta +\sigma ,f)\\
	HF_*^{loc}(\phi^1_{Q+\hat K}) \cong HF_{*+s}^{loc}(\phi^1_{\hat Q+\hat K}).\qquad\qquad\qquad\qquad\qquad
\end{align*}

Now the Hamiltonian $\hat Q+\hat K$ is autonomous so $(\hat H1)$ is still preserved, and we now have that $0\in B^{2n}$ is an isolated fixed point for its time$-t$ flow for $t\in (0,2)$ for the map $\phi^t_{\hat Q+ \hat K}$ which finally proves  $(\hat H2)$.\\

We no longer need to focus on the degenerate parts and non-degenerate parts separately. As such we will refer to $\hat Q+\hat K$ simply as $\hat H$ from this point forward. The only thing we still need to do in order to prove Proposition \ref{autonomous H} is to note we can take $\hat H(0)=0$. Our Hamiltonian is autonomous so $\hat H(0)=c$ and if we replace $\hat H$ by $\hat H-c$ we have achieved this without altering our Hamiltonian vector field or its associated symplectic form $\omega_{\hat H} =d(r-\hat H)\wedge d\theta +\sigma.$ 
\begin{flushright}
$\Box$
\end{flushright}


\begin{bibdiv}
\begin{biblist}*{labels={alphabetic}}


\bib{AGKM}{article}{
   author={Abreu, M.},
   author={Gutt, J.},
   author={Kang, J.}
   author={Macarini, L.}
   title={Two closed orbits for non-degenerate Reeb flows},
   journal={Preprint 2019; available at  \url{https://arxiv.org/abs/1903.06523}; to appear in Math. Proc. Cambridge Philos. Soc.} 
}


\bib{Ar}{book}{
   author={Arnol\cprime d, V. I.},
   author={Kozlov, V. V.},
   author={Ne\u{\i}shtadt, A. I.},
   title={Dynamical systems. III},
   series={Encyclopaedia of Mathematical Sciences},
   volume={3},
   note={Translated from the Russian by A. Iacob},
   publisher={Springer-Verlag, Berlin},
   date={1988},
   pages={xiv+291},
   isbn={3-540-17002-2},
   review={\MR{923953}},
   doi={10.1007/978-3-642-61551-1},
}

\bib{BO}{article}{
   author={Bourgeois, F.},
   author={Oancea, A.},
   title={$S^1$-equivariant symplectic homology and linearized contact
   homology},
   journal={Int. Math. Res. Not. IMRN},
   date={2017},
   number={13},
   pages={3849--3937},
   issn={1073-7928},
   review={\MR{3671507}},
   doi={10.1093/imrn/rnw029},
}

\bib{EKP}{article}{
   author={Eliashberg, Y.},
   author={Kim, S. S.},
   author={Polterovich, L.},
   title={Geometry of contact transformations and domains: orderability
   versus squeezing},
   journal={Geom. Topol.},
   volume={10},
   date={2006},
   pages={1635--1747},
   issn={1465-3060},
   review={\MR{2284048}},
   doi={10.2140/gt.2006.10.1635},
}

\bib{F14}{article}{
   author={Floer, A.},
   title={Witten's complex and infinite-dimensional Morse theory},
   journal={J. Differential Geom.},
   volume={30},
   date={1989},
   number={1},
   pages={207--221},
   issn={0022-040X},
   review={\MR{1001276}},
}

\bib{F15}{article}{
   author={Floer, A.},
   title={Symplectic fixed points and holomorphic spheres},
   journal={Comm. Math. Phys.},
   volume={120},
   date={1989},
   number={4},
   pages={575--611},
   issn={0010-3616},
   review={\MR{987770}},
}

\bib{Gi:CC}{article}{
   author={Ginzburg, V. L.},
   title={The Conley conjecture},
   journal={Ann. of Math. (2)},
   volume={172},
   date={2010},
   number={2},
   pages={1127--1180},
   issn={0003-486X},
   review={\MR{2680488}},
   doi={10.4007/annals.2010.172.1129},
}

\bib{GG:gap}{article}{
   author={Ginzburg, V. L.},
   author={G\"{u}rel, B. Z.},
   title={Local Floer homology and the action gap},
   journal={J. Symplectic Geom.},
   volume={8},
   date={2010},
   number={3},
   pages={323--357},
   issn={1527-5256},
   review={\MR{2684510}},
}

\bib{GG:CCAB}{article}{
   author={Ginzburg, V. L.},
   author={G\"{u}rel, B. Z.},
   title={The Conley conjecture and beyond},
   journal={Arnold Math. J.},
   volume={1},
   date={2015},
   number={3},
   pages={299--337},
   issn={2199-6792},
   review={\MR{3390228}},
   doi={10.1007/s40598-015-0017-3},
}

\bib{GG:convex}{article}{
   author={Ginzburg, V. L.},
   author={G\"{u}rel, B. Z.},
   title={Lusternik-Schnirelmann theory and closed Reeb orbits},
   journal={Math. Z.},
   volume={295},
   date={2020},
   number={1-2},
   pages={515--582},
   issn={0025-5874},
   review={\MR{4100023}},
   doi={10.1007/s00209-019-02361-2},
}

\bib{GGM1}{article}{
   author={Ginzburg, V. L.},
   author={G\"{u}rel, B. Z.},
   author={Macarini, L.},
   title={Multiplicity of closed Reeb orbits on prequantization bundles},
   journal={Israel J. Math.},
   volume={228},
   date={2018},
   number={1},
   pages={407--453},
   issn={0021-2172},
   review={\MR{3874849}},
   doi={10.1007/s11856-018-1769-y},
}

 
\bib{GGM}{article}{
   author={Ginzburg, V. L.},
   author={G\"{u}rel, B. Z.},
   author={Macarini, L.},
   title={On the Conley conjecture for Reeb flows},
   journal={Internat. J. Math.},
   volume={26},
   date={2015},
   number={7},
   pages={1550047, 22},
   issn={0129-167X},
   review={\MR{3357036}},
   doi={10.1142/S0129167X15500470},
}
 
  
\bib{GHHM}{article}{
   author={Ginzburg, V. L.},
   author={Hein, D.},
   author={Hryniewicz, U. L.},
   author={Macarini, L.},
   title={Closed Reeb orbits on the sphere and symplectically degenerate
   maxima},
   journal={Acta Math. Vietnam.},
   volume={38},
   date={2013},
   number={1},
   pages={55--78},
   issn={0251-4184},
   review={\MR{3089878}},
   doi={10.1007/s40306-012-0002-z},
}

\bib{GM}{article}{ 
   author={Ginzburg, V. L.}
   author={Macarini, L.}
   title={Dynamical convexity and closed orbits on symmetric spheres}
   journal={Preprint 2019; available at  \url{https://arxiv.org/abs/1912.04882}; to appear in Duke Math. J} 
}
   
\bib{GS}{article}{
   author={Ginzburg, V. L.},
   author={Shon, J.},
   title={On the filtered symplectic homology of prequantization bundles},
   journal={Internat. J. Math.},
   volume={29},
   date={2018},
   number={11},
   pages={1850071, 35},
   issn={0129-167X},
   review={\MR{3871720}},
   doi={10.1142/S0129167X18500714},
}

\bib{HHM:local Morse}{article}{
   author={Hein, D.},
   author={Hryniewicz, U.},
   author={Macarini, L.},
   title={Transversality for local Morse homology with symmetries and
   applications},
   journal={Math. Z.},
   volume={293},
   date={2019},
   number={3-4},
   pages={1513--1599},
   issn={0025-5874},
   review={\MR{4024596}},
   doi={10.1007/s00209-019-02295-9},
}

\bib{Hi}{article}{
   author={Hingston, N.},
   title={Subharmonic solutions of Hamiltonian equations on tori},
   journal={Ann. of Math. (2)},
   volume={170},
   date={2009},
   number={2},
   pages={529--560},
   issn={0003-486X},
   review={\MR{2552101}},
   doi={10.4007/annals.2009.170.529},
}

\bib{HM:contacthom}{article}{
   author={Hryniewicz, U. L.},
   author={Macarini, L.},
   title={Local contact homology and applications},
   journal={J. Topol. Anal.},
   volume={7},
   date={2015},
   number={2},
   pages={167--238},
   issn={1793-5253},
   review={\MR{3326300}},
   doi={10.1142/S1793525315500119},
}

\bib{HZ}{book}{
   author={Hofer, H.},
   author={Zehnder, E.},
   title={Symplectic invariants and Hamiltonian dynamics},
   series={Birkh\"{a}user Advanced Texts: Basler Lehrb\"{u}cher. [Birkh\"{a}user
   Advanced Texts: Basel Textbooks]},
   publisher={Birkh\"{a}user Verlag, Basel},
   date={1994},
   pages={xiv+341},
   isbn={3-7643-5066-0},
   review={\MR{1306732}},
   doi={10.1007/978-3-0348-8540-9},
}

\bib{Lo}{article}{
   author={Long, Y.},
   title={Index iteration theory for symplectic paths with applications to
   nonlinear Hamiltonian systems},
   conference={
      title={Proceedings of the International Congress of Mathematicians,
      Vol. II},
      address={Beijing},
      date={2002},
   },
   book={
      publisher={Higher Ed. Press, Beijing},
   },
   date={2002},
   pages={303--313},
   review={\MR{1957042}},
   doi={10.1007/978-3-0348-8175-3},
}

\bib{Mc}{article}{
   author={McLean, M.},
   title={Local Floer homology and infinitely many simple Reeb orbits},
   journal={Algebr. Geom. Topol.},
   volume={12},
   date={2012},
   number={4},
   pages={1901--1923},
   issn={1472-2747},
   review={\MR{2994824}},
   doi={10.2140/agt.2012.12.1901},
}

\bib{Oa}{article}{
   author={Oancea, A.},
   title={The K\"{u}nneth formula in Floer homology for manifolds with
   restricted contact type boundary},
   journal={Math. Ann.},
   volume={334},
   date={2006},
   number={1},
   pages={65--89},
   issn={0025-5831},
   review={\MR{2208949}},
   doi={10.1007/s00208-005-0700-0},
}

\bib{Po}{article}{
   author={Po\'{z}niak, M.},
   title={Floer homology, Novikov rings and clean intersections},
   conference={
      title={Northern California Symplectic Geometry Seminar},
   },
   book={
      series={Amer. Math. Soc. Transl. Ser. 2},
      volume={196},
      publisher={Amer. Math. Soc., Providence, RI},
   },
   date={1999},
   pages={119--181},
   review={\MR{1736217}},
   doi={10.1090/trans2/196/08},
}
	
\bib{Sa1}{article}{
   author={Salamon, D.},
   title={Morse theory, the Conley index and Floer homology},
   journal={Bull. London Math. Soc.},
   volume={22},
   date={1990},
   number={2},
   pages={113--140},
   issn={0024-6093},
   review={\MR{1045282}},
   doi={10.1112/blms/22.2.113},
}

\bib{Sa2}{article}{
   author={Salamon, D.},
   title={Lectures on Floer homology},
   conference={
      title={Symplectic geometry and topology},
      address={Park City, UT},
      date={1997},
   },
   book={
      series={IAS/Park City Math. Ser.},
      volume={7},
      publisher={Amer. Math. Soc., Providence, RI},
   },
   date={1999},
   pages={143--229},
   review={\MR{1702944}},
   doi={10.1016/S0165-2427(99)00127-0},
}
	

\bib{SZ}{article}{
   author={Salamon, D.},
   author={Zehnder, E.},
   title={Morse theory for periodic solutions of Hamiltonian systems and the
   Maslov index},
   journal={Comm. Pure Appl. Math.},
   volume={45},
   date={1992},
   number={10},
   pages={1303--1360},
   issn={0010-3640},
   review={\MR{1181727}},
   doi={10.1002/cpa.3160451004},
}

\bib{Vi:GAFA}{article}{
   author={Viterbo, C.},
   title={Functors and computations in Floer homology with applications. I},
   journal={Geom. Funct. Anal.},
   volume={9},
   date={1999},
   number={5},
   pages={985--1033},
   issn={1016-443X},
   review={\MR{1726235}},
   doi={10.1007/s000390050106},
}

\end{biblist}
\end{bibdiv}
\end{document}